\documentclass[11pt]{amsart}
\usepackage{amsmath,amsfonts,amsthm}
\usepackage [dvips]{graphicx}

\newcommand{\quand}{\quad\text{and}\quad}

\theoremstyle{plain}
\newtheorem{main}{Theorem}

\newtheorem{theorem}{Theorem}[section]
\newtheorem{lemma}[theorem]{Lemma}

\newtheorem{proposition}[theorem]{Proposition}
\theoremstyle{remark}

\newcommand{\RR}{{\mathbb R}}

\newcommand{\ZZ}{{\mathbb Z}}
\newcommand{\NN}{{\mathbb N}}

\newcommand{\NF}{{\mathfrak F}}
\renewcommand{\NG}{{\mathfrak G}}

\newcommand{\Diff}{\operatorname{Diff}}

\newcommand{\torus}{{\mathbb T}}
\newcommand{\loc}{{\rm loc}}
\newcommand{\cm}{{\mathfrak m}}
\newcommand{\tcm}{\widetilde\cm}

\newcommand{\hm}{{\widehat m}}

\newcommand{\cB}{{\mathcal B}}
\newcommand{\cC}{{\mathcal C}}

\newcommand{\cE}{{\mathcal E}}
\newcommand{\cF}{{\mathcal F}}

\newcommand{\cO}{{\mathcal O}}
\newcommand{\cP}{{\mathcal P}}

\newcommand{\cU}{{\mathcal U}}
\newcommand{\cV}{{\mathcal V}}
\newcommand{\cW}{{\mathcal W}}
\newcommand{\cX}{{\mathcal X}}
\newcommand{\cY}{{\mathcal Y}}
\newcommand{\cZ}{{\mathcal Z}}

\newcommand{\supp}{{\operatorname{supp\,}}}

\newcommand{\id}{{\operatorname{id}}}

\newcommand{\graph}{\operatorname{graph}}

\newcommand{\Jac}{\operatorname{Jac}}

\newcommand{\tp}{\tilde{p}}

\newcommand{\tcE}{\widetilde{\cE}}

\newcommand{\tf}{{\tilde{f}}}
\newcommand{\tNF}{\widetilde{\NF}}
\newcommand{\tNG}{\widetilde{\NG}}
\newcommand{\tS}{\widetilde{S}}
\newcommand{\trho}{\tilde{\rho}}
\newcommand{\tSigma}{\widetilde{\Sigma}}
\newcommand{\tcF}{\widetilde \cF}
\newcommand{\tm}{\widetilde{m}}
\newcommand{\mE}{m_\cE}
\newcommand{\tmE}{\tm_{\tcE}}
\newcommand{\mS}{m_\Sigma}
\newcommand{\tmS}{\tm_{\tSigma}}

\newcommand{\mSc}{\widehat m}
\newcommand{\tM}{\widetilde{M}}
\newcommand{\tvarphi}{\widetilde{\varphi}}

\title[Absolute continuity, Lyapunov exponents and rigidity]
      {Absolute continuity, Lyapunov exponents and rigidity I : geodesic flows}
\author[A. Avila, M. Viana, A. Wilkinson]
       {A. Avila$^{1,2}$, M. Viana$^2$, and A. Wilkinson$^3$}
\date{\today}

\thanks{$^1$
CNRS UMR 7599, Laboratoire de Probabilit\'es et Mod\`eles
Al\'eatoires, Universit\'e Pierre et Marie Curie, Bo\^\i te Postale
188, 75252 Paris Cedex 05, France.}

\thanks{$2$
IMPA -- Estrada D. Castorina 110, Jardim Bot\^anico, 22460-320 Rio
de Janeiro, Brazil.}

\thanks{$3$
Department of Mathematics, Northwestern University,
2033 Sheridan Road, Evanston, IL 60208-2730, USA}

\email{artur@math.sunysb.edu}
\urladdr{www.impa.br/~avila/}
\email{viana@impa.br}
\urladdr{www.impa.br/~viana/}
\email{wilkinso@math.northwestern.edu}
\urladdr{www.math.northwestern.edu/~wilkinso/}

\begin{document}

\begin{abstract}
We consider volume-preserving perturbations of the time-one map of the geodesic flow
of a compact surface with negative curvature.  We show that if the Liouville
measure has Lebesgue disintegration along the center foliation then the
perturbation is itself the time-one map of a smooth volume-preserving flow,
and that otherwise the disintegration is necessarily atomic.
\end{abstract}

\maketitle


\section{Introduction}\label{s.introduction}

If $\cF$ is a foliation with $C^1$ leaves of a compact manifold $M$,
then for any Borel probability measure $\mu$ on $M$, there is a unique
\emph{disintegration} $\{[\mu_x]:  x \in M\}$ of $\mu$ along the leaves of $\cF$.
The elements $[\mu_x]$ are \emph{projective measures} (that is, equivalence classes of
measures up to scaling) and are defined over a full $\mu$-measure set of $x\in M$.
Each representative $\mu_x$ is supported on the leaf $\cF_x$ of the foliation through $x$.
Locally, a representative measure $\mu_x$ can be described as follows.
One fixes a foliation box $\cB$ for $\cF$ with its foliation by local leaves $\{\cF^{\loc}_{x} : x\in \cB\}$.
In this box, $\mu_x$ is simply the conditional measure of $\mu$ relative to $\cF^{\loc}$,
evaluated at $x$. The conditional measures $\{\mu_x: x\in \cB\}$ are probability measures
satisfying
$$
\mu(A) =  \int_\cB \mu_x(A)\, d\mu(x).
$$
for any Borel set $A\subset \cB$, and they are essentially uniquely defined.

The opposing notions of \emph{Lebesgue disintegration} and \emph{atomic disintegration}
are both well-defined; $\mu$ has Lebesgue disintegration along $\cF$ if for $\mu$-almost every $x$,
any representative of $[\mu_x]$ is equivalent to (i.e. has the same zero sets as)
Riemannian volume on $\cF_x$, and $\mu$ has atomic disintegration if $[\mu_x]$ is an atomic class,
for $\mu$-almost every $x$.  Throughout this paper, we restrict to the case where $\mu$ is a
volume measure on $M$, which we will always denote by $m$.  If $\cF$ is a $C^1$ foliation, then
any volume measure has Lebesgue disintegrations along $\cF$, but the converse is false.
A weaker condition than $C^1$ that implies Lebesgue disintegration of volume is
\emph{absolute continuity}: a foliation is  absolutely continuous if holonomy maps between
smooth transversals send zero volume sets to zero volume sets.

Lebesgue disintegration and in particular absolute continuity have long played a central role in
smooth ergodic theory.  Anosov and Sinai~\cite{An67,AS67} proved in the 60's  that the stable and
unstable foliations of globally hyperbolic (or Anosov) systems are absolutely continuous,
even though they fail to be $C^1$ in general. This was a key ingredient in Anosov's celebrated
proof~\cite{An67} that the geodesic flow for any compact, negatively curved manifold is ergodic.

\subsection{Perturbations of the time-one map of a geodesic flow}

Let $\varphi_t:  T^1 S \to T^1 S$ be the geodesic flow on the unit tangent bundle to a closed,
negatively curved surface $S$.  We consider a discretization of this flow, namely its time-one map
$\varphi_1$, and examine the properties of all diffeomorphisms $f$ that are $C^1$-close to $\varphi_1$.

It follows from the work of Hirsch, Pugh, and Shub~\cite{HPS77} that for any such perturbation $f$
of $\varphi_1$, there exists an $f$-invariant \emph{center foliation} $\cW^c = \cW^c(f)$ with smooth leaves,
that is homeomorphic to the orbit foliation $\cO$ of $\varphi_t$.  Moreover, the homeomorphism
$h: T^1 S \to T^1 S$ sending $\cW^c$ to $\cO$ can be chosen close the identity.

The original orbit foliation $\cO$ of $\varphi_t$ is smooth, and hence volume has Lebesgue disintegration
along $\cO$-leaves.  If the perturbation $f$ happens to be the time-one map of a smooth flow,
then $\cW^c$ is the orbit foliation for that flow, and volume has Lebesgue disintegration along $\cW^c$.
In general, however, a perturbation $f$ of $\varphi_1$ has no reason to embed in a smooth flow,
and one can ask whether the disintegration of volume along $\cW^c$-leaves is Lebesgue, atomic, or neither.
We obtain a complete answer to this question when $f$ preserves volume.

\begin{main}\label{t=geodesic}
Let $\varphi_t: T^1 S\to T^1 S$ be the geodesic flow for a closed negatively curved surface $S$ and
let $m$ be the $\varphi_t$-invariant Liouville probability measure.

There is a $C^1$-open neighborhood
$\cU$ of $\varphi_1$ in the space $\Diff^\infty_m(T^1 S)$ of $m$-preserving
diffeomorphisms of $T^1 S$ such that for each $f\in \cU$:
\begin{enumerate}
\item there exists $k\ge 1$ and a full $m$-measure set $Z\subset T^1 S$ that intersects every
center leaf in exactly $k$ orbits of $f$,
\item or $f$ is the time-one map of an $m$-preserving $C^\infty$ flow.
\end{enumerate}
In case (1) $m$ has atomic disintegration, and in case (2) it has Lebesgue disintegration
along the center foliation $\cW^c(f)$.
\end{main}

Theorem~\ref{t=geodesic} gives conditions under which one can recover the action of a Lie group
(in this case $\RR$) from that of a discrete subgroup (in this case $\ZZ$).
These themes have arisen in the related context of measure-rigidity for algebraic partially
hyperbolic actions by Einsiedler, Katok, Lindenstrauss~\cite{EKL06}.
It would be interesting to understand more deeply the connections between these works.
Results of a similar flavor to Theorem~\ref{t=geodesic} but for the stable and unstable foliations
of Anosov diffeomorphisms and flows have been proved by Benoist, Foulon and Labourie \cite{BFL92,BL93}.

\subsection{Lyapunov exponents and absolute continuity}

The hidden player in Theorem~\ref{t=geodesic} is the concept of \emph{center Lyapunov exponents}.
A real number $\chi$ is a \emph{center Lyapunov exponent} of the partially
hyperbolic diffeomorphism $f: M\to M$ at $x\in M$ if there exists a nonzero vector
$v\in E^c_x$ such that
\begin{equation}\label{e=lyaplim}
\lim_{n\to\infty} \frac{1}{n} \log \|Df^n(v)\| = \chi.
\end{equation}
If $f$ preserves $m$, then Oseledec's theorem implies that the limit in \eqref{e=lyaplim}
exists for each $v\in E^c_x$, for $m$-almost every $x$.
When the dimension of $E^c(f)$ is $1$, the limit in \eqref{e=lyaplim} depends only on $x$,
and if in addition $f$ is ergodic with respect to $m$, then the limit takes a single
value $m$-almost-everywhere. When we refer to a \emph{a center exponent with respect to volume},
we mean a value in \eqref{e=lyaplim} assumed on a positive volume set, and by
\emph{the center exponent with respect to volume} we mean a (the) value assumed almost everywhere.
Deep connections between Lyapunov exponents and geometric properties of invariant measures have
long been understood \cite{Pes76,Pes77,Le84a,LY85a,LY85b,Ka80,BPS99}.
In the context of partially hyperbolic systems, some of these connections have come
to light more recently.

Absolute continuity holds in great generality for the stable and unstable foliations of partially
hyperbolic systems~\cite{BP74, PSh72}, and for Pesin stable and unstable laminations of non-uniformly
hyperbolic systems~\cite{Pes76} (see also Pugh, Shub~\cite{PSh89} for the non-conservative case).
On the other hand, and in sharp contrast, Shub, Wilkinson~\cite{SW00} showed that
\emph{center} foliations of partially hyperbolic systems are, often, \emph{not} absolutely continuous.
What is more, Ruelle, Wilkinson~\cite{RW01} showed that, in a similar setting, the disintegration
of volume along center leaves is atomic, supported on finitely many points.

The mechanism behind these results is nonvanishing center exponents: for each $f$ in the open set
of ergodic diffeomorphisms $\cV\subset \Diff^\infty_\omega(\torus^3)$ constructed in \cite{SW00},
the center Lyapunov exponent with respect to volume is nonzero.
The examples in $\cV$ are obtained by perturbing the trivial extension of a hyperbolic
automorphism of $\torus^2$ on $\torus^3 = \torus^2\times \RR/\ZZ$.
By \cite{HPS77}, the center foliation $\cW^c(f)$ for each $f\in \cV$ is
homeomorphic to the trivial $\RR/\ZZ$ fibration of $\torus^3 = \torus^2\times \RR/\ZZ$;
in particular, the center leaves are all compact.  The almost everywhere exponential growth
associated with nonzero center exponents is incompatible with the compactness of the
center foliation, and so the full volume set with positive center exponent must meet
almost every leaf in a zero set (in fact a finite set) for these examples. In general,
for conservative systems with compact one-dimensional leaves, absolute continuity cannot occur
unless the center Lyapunov exponent vanishes, and this is a kind of codimension-one condition.
Absolute continuity is much more common among dissipative systems, as observed by Viana, Yang~\cite{VY1}.

Similar results hold for perturbations of the time-one map of volume-preserving Anosov flows:
there exist opens sets of perturbations with nonvanishing center exponents
(Dolgopyat~\cite{Dol04}), and the results in \cite{RW01} also imply that volume must have
atomic disintegration for these examples.

The heart of understanding the general perturbation of these and similar examples, then,
is to see what happens when the center Lyapunov exponents \emph{vanish}.
For this, we use tools that originate in the study of random matrix products.
The general theme of this work, summarized by Ledrappier in \cite{Le86} is that
``entropy is smaller than exponents, and entropy zero implies  deterministic."
Original results concerning the Lyapunov exponents of random matrix products,
due to Furstenberg, Kesten~\cite{FK60, Fu63}, Ledrappier~\cite{Le86}, and others,
have been extended in the past decade to deterministic products of linear cocycles
over hyperbolic systems by Bonatti, Gomez-Mont, Viana~\cite{BGV03,BoV04,Almost}.
The Bernoulli and Markov measures associated to random products in those earlier
works are replaced in the newer results by invariant measures for the hyperbolic
system carrying a suitable product structure.

Recent work of Avila, Viana~\cite{AV3} extends this hyperbolic theory from linear
to \emph{smooth} (diffeomorphism) cocycles, and we use these results in a central way.
Also important for our proofs here are the results of Avila, Santamaria, Viana~\cite{ASV}
for cocycles over volume-preserving partially hyperbolic systems, both linear and
smooth.

The ideas introduced in this work have already given rise to further applications
in distinct settings:  the study of measures of maximal entropy \cite{HHTU10b} and
physical measures \cite{VY1} for partially hyperbolic diffeomorphisms with compact
1-dimensional center foliations.

\section{Preliminaries}

We start by recalling a few useful facts concerning foliations and
partially hyperbolic diffeomorphisms.

\subsection{Foliations}\label{ss=fol}

Let $M$ be a manifold of dimension $d\ge 2$. A \emph{foliation (with $C^r$ leaves)}
is a partition $\cF$ of the manifold $M$ into $C^r$ submanifolds of dimension $k$,
for some $0 < k <d$ and $1 \le r \le \infty$, such that for every $p\in M$ there
exists a continuous local chart
$$
\Phi: B_1^k \times B_1^{d-k} \to M \quad\text{($B_1^m$ denotes the unit ball in $\RR^m$)}
$$
with $\Phi(0,0)=p$ and such that the restriction to every horizontal $B_1^k \times \{\eta\}$ is a
$C^r$ embedding depending continuously on $\eta$ and whose image is contained in some $\cF$-leaf.
The image of such a chart $\Phi$ is a \emph{foliation box} and the $\Phi(B_1^k \times \{\eta\})$
are the corresponding \emph{local leaves}.

\subsection{Partially hyperbolic diffeomorphisms}\label{ss.ph}

We say that a diffeomorphism $f: M\to M$ of a compact Riemannian manifold $M$
is \emph{partially hyperbolic} if there exists a continuous, $Df$-invariant
splitting $T M = E^u \oplus E^c \oplus E^s$, into nonzero bundles, and a
positive integer $k$ such that for every $x \in M$,
\begin{equation*}
\begin{aligned}
\|(Df^k \mid E^u(x))^{-1}\|^{-1} > & 1 > \|Df^k \mid E^s(x)\| \, , \\
\|(Df^k \mid E^u(x))^{-1}\|^{-1} > & \|Df^k \mid E^c(x)\| \geq \\
\ge & \|(Df^k \mid E^c(x))^{-1}\|^{-1} > \|Df^k \mid E^s(x)\| \, .
\end{aligned}
\end{equation*}
The time-one map of an Anosov flow is partially hyperbolic and, since partial
hyperbolicity is a $C^1$-open property, so are its perturbations. Since the geodesic
flow for any closed, negatively-curved manifold is Anosov, the maps
considered in this paper are all partially hyperbolic.
For a discussion of partial hyperbolicity, with examples and open questions,
see \cite{BPSW,PSh04,HHU07s}.

The stable and unstable bundles $E^s$ and $E^u$ of a partially hyperbolic diffeomorphism
are always uniquely integrable, tangent to stable and unstable foliations,
$\cW^s$ and $\cW^u$ respectively. The center bundle $E^c$ is not always integrable
(see \cite{HHUcoh}), but in many examples of interest, such as the time-one map of an
Anosov flow and its perturbations, $E^c$ is tangent to a foliation $\cW^c$,
as are the bundles $E^{cs}=E^c\oplus E^s$ and $E^{cu}=E^c\oplus E^u$.
We say that a partially hyperbolic diffeomorphism $f$ is \emph{dynamically coherent} if
there exist $f$-invariant \emph{center stable} and \emph{center unstable} foliations
$\cW^{cs}$ and $\cW^{cu}$, tangent to the bundles $E^{cs}$ and $E^{cu}$, respectively;
intersecting their leaves one obtains an invariant center foliation $\cW^c$ as well.
Most of the facts here are proved in \cite{HPS77}. More detailed discussions can be
found in \cite{Beyond}, \cite{BW08b} and \cite{BW10}.
It is not known whether every perturbation of a dynamically coherent diffeomorphism
is dynamically coherent, but this does hold for systems that are \emph{plaque expansive}.

The notion of plaque expansiveness was introduced by Hirsch, Pugh, and Shub~\cite{HPS77},
who proved among other things that any perturbation of a plaque expansive diffeomorphism
is dynamically coherent. Roughly, $f$ is plaque expansive if pseudo orbits that respect
local leaves of the center foliation cannot shadow each other too closely (in the case of
Anosov diffeomorphisms, plaque expansiveness is the same as expansiveness, which is automatic).
Plaque expansiveness holds in a variety of natural settings;
in particular if $f$ is dynamically coherent, and either $\cW^c$ is a $C^1$ foliation
or the restriction of $f$ to $\cW^c$ leaves is an isometry, then $f$ is plaque expansive,
and so every $C^1$ perturbation of $f$ is dynamically coherent.
Moreover, plaque expansive systems enjoy the previously mentioned stability property:
the center foliations of any two perturbations are homeomorphic via a map that intertwines
the dynamics on the space of center leaves.

Because they are uniformly contracted/expanded by the dynamics, the leaves of stable and
unstable foliations are always contractible; this is not the case for center foliations.
One illustration is the previously mentioned example of the time-one map of an Anosov flow,
for which the center foliation $\cO$ has both compact leaves (corresponding to periodic
orbits of the flow) and non-compact ones.

If $f$ is dynamically coherent, then each leaf of $\cW^{cs}$ is simultaneously subfoliated by
the leaves of $\cW^c$ and by the leaves of $\cW^s$. Similarly $\cW^{cu}$ is subfoliated
by $\cW^c$ and $\cW^u$. This implies that for any two points $x,y\in M$ with $y\in \cW^s_x$
there is a neighborhood $U_x$ of $x$ in the leaf $\cW^c_x$ and a homeomorphism
$h^s_{x,y}: U_x \to \cW^c_y$ with the property that $h^s_{x,y}(x) = y$ and in general
$$
h^s_{x,y}(z) \in \cW^s_z \cap \cW^c_{\loc, y}.
$$
We refer to $h^s_{x,y}$ as a \emph{(local) stable holonomy map}.
We similarly define unstable holonomy maps between local center leaves. We note that,
because the leaves of stable and unstable foliation are contractible,
the local holonomy maps $h^\ast_{x,y}$ for $\ast\in\{s,u\}$ are well-defined and are uniquely
defined as germs by the endpoints $x,y$.

We say that $f$ \emph{admits global stable holonomy maps} if for every  $x,y\in M$ with
$y\in \cW^s_x$ there exists a homeomorphism $h^s_{x,y}: \cW^c_x \to \cW^c_y$ with the
property that $h^s_{x,y}(x) = y$ and in general $h^s_{x,y}(z) \in \cW^s_z\cap \cW^c_y$.
Since global stable holonomy maps must agree locally with local stable holonomy, we use the
same  notation $h^s_{x,y}$ for both local and global. We similarly define global unstable
holonomy maps and say that $f$ \emph{admits global $su$-holonomy maps} if it admits both
global stable and unstable holonomy. Note that if $f$ admits global $su$-holonomy,
then all leaves of $\cW^c$ are homeomorphic.

Given $r>0$, we say that $f$ is \emph{$r$-bunched} if there exists $k\geq 1$ such that:
\begin{equation}\label{eq.rbunched}
\begin{aligned}
& \sup_p \| D_p f^k \mid E^s \| \, \|(D_p f^k \mid E^c)^{-1}\|^r < 1,\\
& \sup_p \|(D_p f^k \mid E^u)^{-1} \| \, \|D_p f^k \mid E^c\|^r < 1, \\
& \sup_p \|D_p f^k \mid E^s \| \, \|(D_p f^k \mid E^c)^{-1}\| \, \|D_p f^k \mid E^c\|^r < 1,\\
& \sup_p \|(D_p f^k \mid E^u)^{-1}\| \, \|D_p f^k \mid E^c\| \, \|(D_p f^k \mid E^c)^{-1}\|^r < 1.
\end{aligned}
\end{equation}
When $f$ is $C^r$ and dynamically coherent, these inequalities ensure that the leaves of
$\cW^{cs}$, $\cW^{cu}$, and $\cW^c$ are $C^r$. If $f$ is $C^{r+1}$ and dynamically coherent
they also imply that the local stable and  local unstable holonomies are $C^r$ local
diffeomorphisms. See Pugh, Shub, Wilkinson~\cite{PSW97,Wliv}.  We say that $f$ is {\em center bunched} if it is $1$-bunched.  If $E^c$ is one-dimensional, then $f$ is automatically center bunched.  For a fixed $r$,
the $r$-bunching property is $C^1$ open: any sufficiently $C^1$ small perturbation of an $r$-bunched diffeomorphism is $r$-bunched.

The ergodic theoretic properties of center bunched partially hyperbolic diffeomorphisms are in many ways well understood.  The state of the art is the following result.
\begin{theorem}\label{t.BW}\cite{BW10} Let $f$ be $C^2$, volume preserving, partially hyperbolic and center bunched.
If $f$ is (essentially)  accessible, then $f$ is ergodic with respect to the volume measure.
\end{theorem}
A partially hyperbolic diffeomorphism is called \emph{accessible} if any
two points in the ambient manifold may be joined by an \emph{$su$-path}, that is, a piecewise smooth path
such that every leg is contained in a single leaf of $\cW^s$ or a single leaf of $\cW^u$. More generally,
the diffeomorphism is \emph{essentially accessible} if, given any two sets with positive volume,
one can join some point of one to some point of the other by an $su$-path.  Pugh and Shub~\cite{PSh96}
have conjectured that accessibility holds for a $C^r$ open and dense subset of the partially hyperbolic diffeomorphisms, volume-preserving or not.

Concerning the conjecture, it was shown by Dolgopyat, Wilkinson~\cite{DW03} that accessibility holds for a $C^1$-open and dense subset of all partially hyperbolic diffeomorphisms, volume preserving or not.
Moreover, Didier~\cite{Di03} proved that accessibility is $C^1$-open for systems with $1$-dimensional
center bundle; that is accessibility implies {\em stable accessibility} for such systems.
More recently, Rodriguez Hertz, Rodriguez Hertz, Ures~\cite{HHU08b} checked the complete
conjecture for conservative systems whose center bundle is one-dimensional: accessibility is
$C^r$-dense among $C^r$ partially hyperbolic diffeomorphisms, for any $r\ge 1$.
A version of this statement for non-conservative diffeomorphisms was obtained in~\cite{BHHTU}.
It remains open whether $C^r$-density still holds when $\dim E^c >1$.

Returning to the context of this paper, stable accessibility has been shown to hold for the time-one map of any geodesic flow in negative curvature \cite{KK96} and more generally for the time-one map of any mixing Anosov flow \cite{BPW00}.  Hence the map $\varphi_1$ in Theorem~A -- and in fact any $C^1$ small perturbation of $\varphi_1$ --  is accessible.  Combining this fact with Theorem~\ref{t.BW}, we see
$\varphi_1$ is {\em stably ergodic}: any $C^2$, volume preserving diffeomorphism that is sufficiently $C^1$ close to $\varphi_1$ is ergodic with respect to the volume measure.

\section{Disintegration of measure}\label{s.measuretheoretic_preliminaries_one}

We begin with a general discussion of disintegration of measures.

\subsection{Measurable partitions and disintegration of measure}\label{ss.disintegration}

Let $Z$ be a polish metric space, let $\mu$ be a finite Borel measure on $Z$,
and let $\cP$ be a partition of $Z$ into measurable sets. Denote by $\hat\mu$ the
induced measure on the $\sigma$-algebra generated by $\cP$,
which may be naturally regarded as a measure on $\cP$.

A \emph{system of conditional measures} (or a \emph{disintegration}) of $\mu$ with respect to
$\cP$ is a family $\{\mu_P\}_{P\in \cP}$ of probability measures on $Z$ such that
\begin{enumerate}
\item $\mu_P (P) = 1$ for $\mu$-almost every $P \in \cP$;
\item Given any continuous function $\psi: Z\to \RR$, the function $P\mapsto
\int\psi \,d\mu_P$ is
measurable, and
$$\int_M\psi\,d\mu = \int_\cP\left( \int \psi \,d\mu_P \right) d\hat\mu(P).
$$
\end{enumerate}

It is not always possible to disintegrate a probability measure with respect to a partition --
we discuss examples below -- but disintegration is always possible if $\cP$ is a measurable partition.
We say that $\cP$ is a \emph{measurable partition} if there exist measurable subsets
$E_1,E_2,\ldots, E_n\ldots $ of $Z$ such that
\begin{equation}\label{eq.measurablpartition}
\cP = \{E_1, Z\setminus E_1\} \vee \{E_2, Z\setminus E_2\} \vee \cdots \quad\quad\mod 0.
\end{equation}
In other words, there exists a full $\mu$-measure subset $F_0\subset Z$ such that,
for any atom $P$ of $\cP$, we have
$$
P\cap F_0 =  E_1^\ast\cap E_2^\ast \cap \cdots \cap F_0
$$
where $E_i^\ast$ is either $E_i$ or $Z\setminus E_i$, for $i\geq 1$.
Our interest in measurability of a partition derives from the following fundamental result.

\begin{theorem}[Rokhlin \cite{Ro52}] If $\cP$ is a measurable partition,
then there exists a system of conditional measures relative to $\cP$.
It is essentially unique in the sense that two such systems coincide in a set of full $\hat\mu$-measure.\end{theorem}

\subsection{Disintegration of measure along foliations with noncompact leaves}

The disintegration theorem of Rokhlin~\cite{Ro52} does not apply directly when a foliation has
a positive measure set of noncompact leaves.
Instead, one must consider disintegrations into \emph{measures defined up to scaling},
that is, equivalence classes where one identifies any two (possibly infinite) measures that differ only
by a constant factor. Here we present this theory in a fairly general setting.

Let $M$ be a manifold of dimension $d\ge 2$, and let
$m$ be a locally finite measure on $M$. Let $\cB$ be any (small) foliation box.
By Rokhlin~\cite{Ro52}, there is a disintegration $\{m_x^\cB:  x \in \cB\}$ of the restriction of
$m$ to the foliation box into conditional probabilities along the local leaves, and this disintegration
is essentially unique. The crucial observation is that conditional measures corresponding to
different foliation boxes coincide on the intersection, up to a constant factor.

\begin{lemma}\label{eq.uniquedisintegration}
For any foliation boxes $\cB$ and $\cB'$ and for $m$-almost every $x\in\cB\cap\cB'$,
the restrictions of $m_x^{\cB}$ and $m_x^{\cB'}$ to $\cB\cap\cB'$ coincide up to a
constant factor.
\end{lemma}

\begin{proof}
Let $\Sigma$ be a cross-section to $\cB$, that is, a submanifold of dimension $d-k$ intersecting
every local leaf at exactly one point. Let $\mu_B$ be the measure on $\Sigma$ obtained by projecting
$(m\mid\cB)$ along the local leaves. Consider any $\cC\subset\cB$ and let $\mu_\cC$ be the image of
$(m\mid\cC)$ under the projection along the local leaves. The Radon-Nikodym derivative
$$
\frac{d\mu_\cC}{d\mu_\cB} \in (0,1] \quad\text{at $\mu_\cC$-almost every point.}
$$
For any measurable set $E\subset C$,
$$
m(E) 
 = \int_\Sigma m^\cB_\xi(E)\,d\mu_\cB(\xi)
 = \int_\Sigma m^\cB_\xi(E) \frac{d\mu_\cB}{d\mu_\cC}(\xi) \,d\mu_\cC(\xi)\,
$$
By essential uniqueness, this proves that the disintegration of $(m\mid\cC)$ along the local leaves is
given by
\begin{equation}\label{eq.rescaling}
m^\cC_x = \frac{d\mu_\cB}{d\mu_\cC}(\xi)(m^\cB_x\mid\cC)\quad\text{for $m$-almost every $x\in\cC$}
\end{equation}
where $\xi$ is the point where the local leaf through $x$ intersects $\Sigma$. Now we take
$\cC=\cB\cap\cB'$. Using \eqref{eq.rescaling} twice we get
$$
\frac{d\mu_\cB}{d\mu_\cC}(\xi)(m^\cB_x\mid\cC) = \frac{d\mu_\cB'}{d\mu_\cC}(\xi)(m^{\cB'}_x\mid\cC)
$$
for $m$-almost every $x$. This proves the lemma.
\end{proof}

This implies that there exists a family $\{\cm_x:  x\in M\}$ where each $\cm_x$ is a measure defined up
to scaling with $\cm_x(M\setminus\cF_x)=0$, the function $x\mapsto\cm_x$ is constant on the leaves of
$\cF$, and the conditional probabilities $m_x^\cB$ along the local leaves of any foliation box $\cB$
coincide almost everywhere with the normalized restrictions of the $\cm_x$ to the local leaves of $\cB$.
It is also clear from the arguments that such a family is essentially unique. We call it
the disintegration
of $m$ and refer to the $\cm_x$ as conditional classes of $m$ along the leaves of $\cF$ .

\subsection{Foliations whose leaves are fixed under a measure-preserving homeomorphism}

Now suppose the foliation $\cF$ is invariant under a homeomorphism $f: M\to M$, meaning that
$f(\cF_x)=\cF_{f(x)}$ for every $x\in M$. Take the measure $m$ to be invariant under $f$.
Then, by essential uniqueness of the disintegration, $f_*(\cm_x)=\cm_{f(x)}$ for almost every $x$.
We are especially interested in the case when $f$ fixes every leaf, that is, when
$f(x)\in\cF_x$ for all $x\in M$.
Then $f_*(\cm_x)=\cm_{f(x)}$ for almost every $x$, which means that every representative $m_x$ of
the conditional class $\cm_x$ is $f$-invariant up to rescaling: $f_*(m_x) = c m_x$ for some $c>0$.
Actually, the scaling factor $c=1$:

\begin{proposition}\label{eq.invarianceofconditionals}
Suppose that $m$ is invariant under a homeomorphism $f: M\to M$ that fixes all the leaves of $\cF$.
Then, for almost all $x\in M$, any representative $m_x$ of the conditional class $\cm_x$ is an
$f$-invariant measure.
\end{proposition}

\begin{proof}
Fix $x_0\in M$ and let $\cB$ be a foliation box containing both $x_0$ and $f(x_0)$. Let $\Sigma$ be a
cross-section to $\cB$ and let $\mu_\cB$ be the image of $(m\mid \cB)$ under the projection
$p: \cB\to\Sigma$ along the local leaves. Choose representatives $m_x$ of the conditional classes scaled
so that the restriction of $m_x$ to the local leaf $\cF^\cB_x$ through every $x\in\cB$ is a probability.
Then
\begin{equation}\label{eq.restriction}
m_x^\cB = (m_x\mid\cF^\cB_x).
\end{equation}
Now let $\cB_0$ be a foliation box containing $x_0$, small enough that $\cB_0$ and $\cB_1=f(\cB_0)$
are both contained in $\cB$. Note that $(m\mid\cB_1) = f_*(m \mid \cB_0)$, because $m$ is invariant,
and $p\circ f = p$, because all the leaves are fixed by $f$. Thus,
$p_*(m\mid \cB_0) = p_*(m \mid \cB_1)$. We denote this measure by $\nu$. By \eqref{eq.rescaling}
and \eqref{eq.restriction},
$$
m_x^{\cB_0} = \frac{d\mu_\cB}{d\nu}(\xi) (m_x\mid\cF^{\cB_0}_x)
 \quad\text{and}\quad
m_y^{\cB_1} = \frac{d\mu_\cB}{d\nu}(\eta) (m_y\mid\cF^{\cB_1}_y)
$$
for almost every $x\in\cB_0$ and $y\in\cB_1$, where $\xi=p(x)$ and $\eta=p(y)$. On the other hand,
since $f$ maps local leaves of $\cB_0$ to local leaves of $\cB_1$, the images of the $m_x^{\cB_0}$ under $f$
define a disintegration of $m\mid\cB_1$ along the leaves. By essential uniqueness, it follows that
$$
\frac{d\mu_\cB}{d\nu}(\eta) (m_y\mid\cF^{\cB_1}_y)
 = m_y^{\cB_1}
 = f_*(m_x^{\cB_0})
 = \frac{d\mu_\cB}{d\nu}(\xi)f_*(m_x\mid\cF^{\cB_0}_x)
$$
for almost every $x\in\cB_0$, where $y=f(x)$. Since $m_x=m_y$ and $\xi=\eta$, it follows that
$(m_x\mid\cF^{\cB_1}_{f(x)}) = f_*(m_x\mid\cF^{\cB_0}_x)$ for almost every $x\in\cB_0$.
This proves that $m_x$ is indeed invariant (the scaling factor is $1$) for almost every point
in $\cB_0$. Covering $M$ with such foliation boxes one gets the conclusion of the proposition.
\end{proof}

\subsection{Absolute continuity}\label{ss.absolutecontinuity}

This is analyzed in a lot more detail in \cite{PVW}. Here we just present a few facts that
are useful for what follows. As above, let $M$ be a Riemannian manifold.
Let $\lambda_\Sigma$ denote the volume measure induced by the Riemann metric on a $C^1$
submanifold $\Sigma$ of $M$.

The classical definition of absolute continuity (\cite{An67,AS67}) goes as follows.
A foliation $\cF$ on $M$ is \emph{absolutely continuous} if every holonomy map $h_{\Sigma,\Sigma'}$
between a pair of smooth cross-sections $\Sigma$ and $\Sigma'$ is absolutely continuous,
meaning that, the push-forward $(h_{\Sigma,\Sigma'})_* \lambda_\Sigma$ is absolutely continuous
with respect to $\lambda_{\Sigma'}$. Reversing the roles of the cross-sections, one sees that
$(h_{\Sigma,\Sigma'})_* \lambda_\Sigma$ is actually equivalent to $\lambda_{\Sigma'}$.

Here it is convenient to introduce the following weaker notion.
We say that \emph{volume has Lebesgue disintegration along $\cF$-leaves} if given any measurable
set $Y\subset M$ then $m(Y)=0$ if and only if for $m$-almost every $z \in M$ the leaf $L$ through
$z$ meets $Y$ in a zero $\lambda_L$-measure set. In other words, for almost every leaf $L$, the
conditional measure $m_L$ of $m$ along the leaf is equivalent to the Riemann measure $\lambda_L$
on the leaf.

\begin{lemma}\label{l.absac2leafac}
If $\cF$ is absolutely continuous then volume has Lebesgue disintegration along $\cF$-leaves.
\end{lemma}

\begin{proof}
Fixing a smooth foliation transverse to $\cF$, and using the fact that the holonomies are
absolutely continuous, one defines a local change of coordinates
$$
(x,y) \mapsto (x, h(0,x)(y))
$$
that rectifies the leaves of $\cF$ and transforms $m$ to a measure of the form $J(x,y) dx dy$
with $J>0$. Lebesgue disintegration is clear in these coordinates.
\end{proof}

The converse is false: one can destroy absolute continuity of holonomy at a single transversal
while keeping Lebesgue disintegration of volume (this is an exercise in Brin, Stuck~\cite{BS02}).

\begin{lemma}\label{l.classic} Let $f\colon M \to M$ be $C^2$ and partially hyperbolic.
The foliations $\cW^{s}(f)$ and $\cW^{u}(f)$ are absolutely continuous and, hence,
volume has Lebesgue disintegration along $\cW^{s}(f)$ and $\cW^{u}(f)$-leaves.
\end{lemma}

\begin{proof}
This is a classical fact going back to Brin, Pesin~\cite{BP74}. They use the uniform
definition of partial hyperbolicity. A proof for the more general pointwise definition
that we consider here can be found in \cite{AbV}.
\end{proof}

\section{Lyapunov exponents and an Invariance Principle}\label{s.measuretheoretic_preliminaries_two}

In this section, we describe the main results we use concerning Lyapunov exponents and
invariant measures of smooth cocycles.

Let $\NF: \cE\to\cE$ be a continuous smooth cocycle over $f$, in the sense of \cite{ASV,AV3}.
This means that $\pi\colon \cE\to M$ is a continuous fiber bundle with fibers modeled on some
Riemannian manifold and $\NF$ is a continuous fiber bundle morphism over a Borel measurable map
$f\colon M\to M$ acting on the fibers by diffeomorphisms with uniformly bounded derivative.
Let $\hat  \mu$ be an $\NF$-invariant probability measure on $\cE$ that projects to an
$f$-invariant measure $\mu$.  We denote by $\cE_x$ the fiber $\pi^{-1}(x)$
and by $\NF_x\colon  \cE_x \to \cE_{f(x)}$ the induced diffeomorphism on fibers.

We say that a real number  $\chi$  is a {\em fiberwise exponent of $\NF$ at $\xi\in \cE$} if
there exists a nonzero vector $v\in T_\xi \cE_{\pi(\xi)}$ in the tangent space to the fiber at
$\xi$ such that
$$
\lim_{n\to \infty} \frac{1}{n}\log \| D_\xi \NF^n(v) \|  = \chi.
$$
By Oseledec's theorem, this limit $\chi(\xi, v)$ exists  for $\hat \mu$-almost every $\xi\in \cE$
and every nonzero $v\in T_\xi \cE_{\pi(\xi)}$, and it takes finitely many values at each such $\xi$.
Let
$$
\bar\chi(\xi) = \sup_{\|v\|=1} \chi(\xi,v) \quad\text{and}\quad \underline\chi(\xi) = \inf_{\|v\|=1} \chi(\xi,v).
$$

The following result follows almost immediately from Theorem II in \cite{RW01} and uses no
assumptions on the base dynamics $f:M\to M$ other than invertibility. The hypothesis on the
fibers can be weakened, but the statement that follows is sufficient for our purposes.

\begin{theorem}\cite{RW01}\label{t.rwmain}
Let $\NF:\cE\to\cE$ be a smooth cocycle over $f$. Assume that the fibers of $\cE$ are compact.
Assume that $\NF$ preserves an ergodic probability measure $\hat\mu$ that projects to an
($f$-invariant, ergodic) probability $\mu$ on $M$ and that $f$ is invertible on a full $\mu$-measure set.
Let $\cX_-$ be the set of $\xi\in \cE$ such that $\bar\chi(\xi)<0$ and $\cX_+$ be the set of
$\xi\in \cE$ such that $\bar\chi(\xi)>0$.

Then both $\cX_-$ and $\cX_+$ coincide up to zero $\hat\mu$-measure subsets with measurable
sets that intersect each fiber of $\cE$ in finitely many points.
\end{theorem}

 The next result, from \cite{AV3, ASV}, treats the possibility that all
fiberwise exponents vanish. It admits more general formulations, but we state it in the
context in which we will use it, namely, when $f$ is a partially hyperbolic diffeomorphism.

We say that $\NF$ \emph{admits a $\ast$-holonomy} for $\ast\in\{s,u\}$ if, for every pair of
points $x,y$ lying in the same $\cW^\ast$-leaf, there exists a H\"older continuous homeomorphism
$H^\ast_{x,y}:  \cE_{x} \to \cE_{y}$ with uniform H\"older exponent, satisfying:
\begin{itemize}
  \item[(i)] $H^\ast_{x,x} =\id$,
 \item[(ii)] $ H^\ast_{x,z} = H^\ast_{y,z} \circ H^\ast_{x,y}$,
 \item[(iii)] $\NF_y \circ H^\ast_{x,y} = H^*_{f(x),f(y)} \circ \NF_x$, and
 \item[(iv)] $(x,y) \mapsto H_{x,y}^\ast(\xi)$ is continuous on the space of
       pairs of points $(x,y)$ in the same local $\cW^\ast$-leaf,
       uniformly on $\xi$.
\end{itemize}

The existence of a $\ast$-holonomy is equivalent to the existence of an
$\NF$-invariant foliation (with potentially nonsmooth leaves) of $\cE$
whose leaves project homeomorphically (in the instrinsic leaf topology) to $\cW^\ast$-leaves in $M$.

A disintegration $\{\hat\mu_x:  x\in M\}$  is \emph{$\ast$-invariant} over a set $X\subset M$,
$\ast\in\{s,u\}$ if the homeomorphism $H_{x,y}^\ast$ pushes $\hat\mu_{x}$ forward to $\hat\mu_{y}$
for every $x,y\in X$ with $y\in \cW^\ast_x$.
We call a set $X\subset M$ \emph{$\ast$-saturated}, $\ast\in\{s,cs, c, cu, u\}$
if it consists of entire leaves of $\cW^\ast$.
Observe that $f$ is accessible if and only if the only nonempty set in $M$ that is both $s$-saturated and
$u$-saturated is $M$ itself.

\begin{theorem}\cite[Theorem~C]{ASV}\label{t.asvmain}
Let $\NF$ be a smooth cocycle over the $C^2$,  volume preserving  partially hyperbolic diffeomorphism $f$.
Assume that $f$ is center bunched and accessible and that $\NF$ preserves a probability measure $\hm$
that projects to the volume $m$.  Suppose  that $\bar\chi(\xi) = \underline\chi(\xi) = 0$ for
$\hm$-almost every $\xi\in \cE$.

Then there exists a continuous  disintegration $\{\hm_x^{su}:  x\in M\}$ of $\hm$ that is invariant under
both $s$-holonomy and $u$-holonomy.
\end{theorem}

Notice that the hypotheses on $f$ in Theorem~\ref{t.asvmain}  coincide with the hypotheses of the ergodicity criterion in Theorem~\ref{t.BW}; they are satisfied by all maps considered in this paper.

\section{Starting the proof of Theorem A}

The proof of Theorem A runs through this and the next two sections.
Here we construct, over every diffeomorphism close to the time-one map, a certain smooth
cocycle $\NF:\cE\to\cE$ with $su$-holonomy, endowed with an invariant measure $\mE$,
whose fiberwise Lyapunov exponent coincides with the center Lyapunov exponent of the diffeomorphism.

Let $S$ be a negatively curved surface and $\varphi_t:M\to M$ be the geodesic flow on the
unit tangent bundle $M=T^1 S$, whose orbits are lifts to $M$ of geodesics in $S$.
The unit tangent bundle $\tM=T^1 \tS$ of the universal cover $\tS$ is a cover
(though not the universal cover) of $M$ and the geodesic flow $\tvarphi_t\colon \tM \to \tM$
covers $\varphi_t$. Since $S$ is negatively curved, the Cartan-Hadamard Theorem implies that
$\tS$ is contractible and the exponential map $\exp_p \colon T_p\tS \to \tS$ is a diffeomorphism
for each $p\in M$. In particular, the orbits of $\tvarphi_t$ are all open, diffeomorphic to $\RR$.

Consider the time-$1$ map $\varphi_1$, and note that $\tvarphi_1$ is a lift of $\varphi_1$.
As explained  in Section~\ref{ss.ph},
the map $\varphi_1$ is partially hyperbolic, center bunched and stably accessible.
Theorem~\ref{t.BW} implies that $\varphi_1$ is stably ergodic.
The foliation $\cW^c(\tvarphi_1)$ by $\tvarphi$-orbits is clearly $\tvarphi_1$-invariant,
and $\tvarphi_1$ acts as a translation by $1$ in each  $\cW^c(\tvarphi_1)$ - leaf.
The foliation $\cW^c(\tvarphi_1)$ is also normally hyperbolic and, being smooth, plaque expansive.
The projection of $\cW^c(\tvarphi_1)$ to $M$ is the center foliation $\cW^c(\varphi_1)$.
It has a natural orientation determined by the vector field $\dot\varphi$.

Let $f\colon M\to M$ be a $C^\infty$ volume-preserving diffeomorphism $C^1$-close to $\varphi_1$.
Then $f$ is partially hyperbolic, center bunched, accessible and ergodic.
In addition, $f$ is dynamically coherent.
Let $\tf:\tM \to \tM$ be the lift of $f$ that is $C^1$ close to $\tilde \varphi_1$.
The lifted foliations $\cW^\ast(\tf)$ are homeomorphic to $\cW^\ast(\tvarphi_1)$, for $\ast\in\{c,cu,cs\}$.
The action of $\tf$ on each leaf of $\cW^c(\tf)$ is uniformly close to a translation by $1$ and,
therefore, is topologically conjugate to a translation.
The leaves of $\cW^{cs}(\tf)$ are bifoliated by the leaves of $\cW^{c}(\tf)$ and $\cW^{s}(\tf)$.
Before perturbation, the $\cW^{s}(\tvarphi_1)$-holonomy maps between center leaves are
orientation-preserving isometries: this follows from the fact that the flow $\varphi_t$ preserves
the stable foliation.

\begin{lemma}\label{l.globalholonomies}
The map $\tf$ admits global $su$-holonomy.
\end{lemma}

\begin{proof}
To check that $\tf$ admits global stable holonomy maps, we must show that for every
$v,v'\in \tM$ with $v'\in \cW^s(\tf)_v$, and for any $w\in  \cW^c(\tf)_v$,
there is a unique point $w'$ in the intersection $ \cW^s(\tf)_w \cap \cW^c(\tf)_{v'}$.
Since $\tf$ acts on center leaves close to a translation by $1$, and uniformly contracts
stable leaves, it suffices to prove this claim for $w$ lying a distance $\leq 2$ from $v$ and
$v'$ a fixed small distance from $v$.  But the claim clearly holds in this case, since the
stable holonomy for $\tf$ between center leaves at a distance $\leq \epsilon$ is
uniformly close to the stable holonomy of $\varphi_1$, which is an isometry.
This proves that $\tf$ has global stable holonomy. The proof for unstable holonomy is analogous.
\end{proof}

The fact that $\tf$ admits global $su$-holonomy allows us to construct a fiber bundle $\tcE$
over $\tM$ whose fibers are leaves of the center foliation $\cW^c(\tf)$, as follows.
For $v,w\in \tM$ and $\ast\in\{s,c,u\}$, we write $v\sim_\ast w$ if $v\in \cW^\ast(\tf)_w$.
Let
$$
\tcE = \{(v,w) \in \tM^2\,\mid\, v\sim_c w\}
$$
and let $\tp_1$, $\tp_2\colon \tcE \to \tM$ be the coordinate projections onto the first and
second $\tM$ factor, respectively.

\begin{lemma}
The projection
$\tp_1\colon \tcE\to \tM $ defines a fiber bundle with the following properties:
\begin{enumerate}
\item $\tp_2$ sends each fiber $\tcE_v=(\tp_1)^{-1}(v)$,  $v\in \tM$ homeomorphically onto the
      center leaf $\cW^c(\tf)_v$;
\item $\tcE$ admits a canonical continuous ``diagonal" section sending each $v\in \tM$ to
      $(v,v)=(\tp_2)^{-1}(v) \cap \tcE_v$.
\end{enumerate}	
\end{lemma}
We remark that the conclusions of this lemma hold with the roles of $\tp_1$ and $\tp_2$ switched.
When we refer to the ``fiber bundle $\tcE$" it is with respect to the first projection $\tp_1$.

\begin{proof}
Given any $v\in \tM$ and $v'$ in a small neighborhood $U$ of $v$ in $\tM$,
define $w$ to be the point in $\cW^{s}_{\loc}(\tf)_v \cap \cW^{cu}_{\loc}(\tf)_{v'}$ and
$w'$ to be the point in $\cW^{u}_{\loc}(\tf)_w \cap \cW^{c}_{\loc}(\tf)_{v'}$.
Notice that $w$ and $w'$ depend continuously on $v'$.
Then $h_{v,v'}=h^u_{w,w'} \circ h^s_{v,w}$ is a homeomorphism from $\cW^c(\tf)_v$ to
$\cW^c(\tf)_{v'}$ that depends continuously on $v'$. It follows that
$$
g_{v,U} : U \times \cW^c(\tf)_v\to \pi^{-1}(U), \quad
(v', \eta) \mapsto (v', h_{v,v'}(\eta))
$$
is a homeomorphism mapping each vertical $\{v'\}\times \cW^c(\tf)_v$ to $(\tp_1)^{-1}(v')$.
This defines on $\tcE$ the structure of a continuous fiber bundle.
It is clear that every fiber $(\tp_1)^{-1}(v)=\{v\}\times\cW^c(\tf)_v$ is mapped
homeomorphically to $\cW^c(\tf)_v$ by the second projection $\tp_2$, as claimed in (1).
The diagonal embedding $\tM \to\tcE$ defines a section as in (2).
\end{proof}

The fundamental group $\pi_1(S)$ acts on $\tM$ by isometries preserving the $\sim_\ast$
equivalence relations:
$$
v\sim_\ast w \,\implies \, \gamma v\sim_\ast \gamma w,\quad
\text{for all $\gamma\in \pi_1(S)$, $v,w\in \tM$ and $\ast\in\{s,u,c\}$}.
$$
Consider the  induced diagonal action of $\pi_1(S)$ on $\tM^2$.
Since this action preserves the $\sim_c$ relation and the product structure,
it preserves the fiber bundle $\tcE$. The stabilizer of each fiber of $\tcE$
under this action is trivial.

There is also a  $\ZZ\times \ZZ$-action $\trho$ on $\tM^2$  commuting with the $\pi_1(S)$-action, defined by
$$
\trho(m,n)(x,y) = (\tf^n(x), \tf^m(y)).
$$
Then $\trho$ also preserves the $\sim_\ast$ equivalence relations and in particular
defines an action on $\tcE$. The action of $\trho(1,0)$ on each $\tp_1$-fiber is
topologically conjugate to a translation, and the action of $\trho(0,1)$  on each
$\tp_2$-fiber is also conjugate to a translation.
Let $\tNF = \trho(1,1)$ and  $\tNG = \trho(0,1)$.  Note that
$$
\tp_1\circ \tNF
 = \tf \circ \tp_1, \quad \tp_2\circ \tNF
  = \tf\circ \tp_2, \quad \tp_1\circ \tNG
   = \tp_1, \quad\hbox{and} \quad \tp_2\circ \tNG
    = {\tf}\circ \tp_2.
$$

Let $\cE$ be the quotient of $\tcE$ by the diagonal $\pi_1(S)$-action.  Denote by
$p_i \colon \cE \to M$, $i=1, 2$ the quotient projections.  The fibers of
$p_1\colon \cE\to M$ are homeomorphic to $\RR$, and for any $v\in M$,
$$
p_2\circ p_1^{-1}(v) = p_1\circ p_2^{-1}(v) = \cW^c(f)_v.
$$
Since $\trho$ commutes with the $\pi_1(S)$-action on $\tcE$, it also induces an action on
the bundle $\cE$, which we denote by $\rho$. Let $\NF = \rho(1,1)$ and $\NG = \rho(0,1)$. Then
$$
p_1\circ \NF = f \circ p_1, \quad p_2\circ \NF = f\circ p_2, \quad p_1\circ \NG
= p_1, \quad\hbox{and} \quad p_2\circ \NG = {f}\circ p_2.
$$

The fiber $\cE_v$ of $\cE$ over $v\in M$ is naturally identified with the leaf
$\cW^c(\tf)_{v'}$ through any lift $v'$ of $v$ to $\tM$. The action of $\NG$ on
this fiber is then naturally identified with the action of $\tf$ on this leaf.
For almost every (for all but countably many) $v\in M$, the leaf $\cW^c(f)_v$ is
noncompact and hence is canonically identified with any lift to $\tM$.
For such $v$, we identify $\cE_v$ with $\cW^c(f)_v$ and the action of $\NG$ on
$\cE_v$ with the action of $f$.

We define $\tNF$-invariant foliations $\tcF^\ast$ of the bundle $\tcE$ whose leaves project
homeomorphically under $\tp_1$ to leaves of $\cW^\ast(\tf)$, as follows:
$$
\begin{aligned}
\tcF^{\ast}_{(v,w)}
 & = \{ (v',w') \in \tcE \mid v'\sim_\ast v \text{ and } w' \sim_\ast w\} \\
\end{aligned}
$$
for $\ast\in\{s,u\}$ (recall Lemma \ref {l.globalholonomies}) and
$$
\tcF^{c}_{(v,w)} = \{ (v',w)  \in \tcE \mid v' \sim_c v \}
                = \{ (v',w) \in \tcE \mid v' \sim_c w \} = (\tp_2)^{-1}(w) .
$$
It follows from the construction that $\tcF^\ast$  is invariant under the action
$\trho$.
Notice that for $\ast \in \{s,u\}$, the leaves of $\tcF^*$ also project homeomorphically under
$\tp_2$ to leaves of $\cW^\ast(\tf)$.

Let $\cF^\ast$ be the induced quotient foliations of $\cE$. Those foliations are clearly $\NF$-invariant.
By definition, for each $\ast\in \{s,u\}$ and every $v,v'$ lying in the same $\cW^\ast$-leaf in $M$,
there exists a holonomy map
\begin{equation}\label{eq.holonomy1}
H^\ast_{v,v'}\colon \cE_v\to \cE_{v'}
\end{equation}
sending $\xi\in \cE_v$ to the unique point $H^\ast_{v,v'}(\xi)$ in the intersection $\cF^\ast(\xi)\cap \cE_{v'}$.
Invariance of the foliations $\cF^\ast$ under $\NF$ implies that for $\ast\in\{s,u\}$, we have:
\begin{equation}\label{eq.holonomy2}
\NF \circ H^{\ast}_{v,v'} =  H^{\ast}_{f(v),f(v')}   \circ \NF.
\end{equation}
In other words, \emph{the cocycle $\NF: \cE \to \cE$ admits $su$-holonomy}.
It will also be useful to consider the \emph{$c$-holonomy}
\begin{equation}\label{eq.holonomy3}
H^c_{v,v'}\colon \cE_v\to \cE_{v'},
\end{equation}
which is given by $(v,w) \mapsto (v',w)$ for every $v,v'$ in the same $\cW^c$-leaf in $M$.
The invariance property \eqref{eq.holonomy2} remains valid for the $c$-holonomy.

\subsection{Constructing a measure on $\cE$}

Denote by $\tm$ the $\pi_1(S)$-invariant lift of $m$ to $\tM$.
It is a $\sigma$-finite, $\tf$-invariant measure whose restriction to any $\pi_1(S)$
fundamental domain is a probability measure that projects to $m$ on $M$.
We next construct a Radon measure $\tmE$ on $\tcE$ that projects to $\tm$,
whose restriction to a $\langle \tNG,\pi_1(S) \rangle$ fundamental domain is a probability measure
and which is $\trho$-invariant and $\pi_1(S)$-invariant.

For $v$ and $w$ lying in the same $\cW^c(\tf)$-leaf,  we denote by $[v,w)$ the
positively oriented arc in  $\cW^c(\tf)_v$ from $v$ to $w$.
Let $\{\tcm_v\}$ be a disintegration of $\tm$ along $\cW^c(\tf)$-leaves.
For each $v\in \tM$, choose a representative $\tm_v$ of the conditional class $\tcm_v$
normalized by
\begin{equation}\label{eq.normalization}
\tm_v\big([v,\tf(v))\big)=1.
\end{equation}
(By $\tf$ invariance, the class $\tcm_v$ is nonvanishing over a fundamental domain
of the action of $\tf$ on $\cW^c(\tf)_v$, for $\tm$-almost every $v$, so that
\eqref{eq.normalization} does make sense.) This choice of a normalization immediately
implies that
\begin{equation} \label{eq.normalization3}
\tf_* \tm_v = \tm_{\tf(v)}.
\end{equation}
Moreover, using Proposition~\ref{eq.invarianceofconditionals},
\begin{equation*}
\tm_v\big([w,\tf(w))\big)
= \tm_v\big([v,\tf(v))\big) = 1
\quad \text{for every $w \in \cW^c(\tf)_v$,}
\end{equation*}
so that we have
\begin{equation} \label{eq.normalization4}
\tm_w=\tm_v \quad \text{for every $w \in \cW^c(\tf)_v$.}
\end{equation}
Then $\tmE=\tm_v d\tm(v)$ defines a Radon measure on $\tcE$ that is
$\tNF$-invariant, by the choice of normalization \eqref{eq.normalization},
$\tNG$-invariant, because of property \eqref{eq.normalization3}, and
$\pi_1(S)$-invariant, since $\tm$ is.

The measure $\tmE$ projects to a measure $\mE$ on $\cE$; writing $\mE= m_v dm(v)$,
the conditional measure $m_v$ of $\mE$ on each fiber $\cE_v$ is naturally identified
with the measure $\tm_{v'}$, where $v'$ is any lift of $v$ to $\tM$. In particular,
\begin{equation} \label{eq.normalization6}
m_w=m_v \quad \text{for every $w \in \cW^c(f)_v$ and}
\end{equation}
\begin{equation} \label{eq.normalization5}
f_*m_v = m_{f(v)} = m_v \quad\text{for every $v\in M$}.
\end{equation}
Property \eqref{eq.normalization6} may be rewritten as $(H^c_{v,w})_*m_v = m_w$ for
every $v, w$ in the same center leaf; we say that the family $\{m_v\}$ is {\em invariant
under $c$-holonomy}.
For those $v\in M$ for which the center leaf is noncompact, the measure $m_v$ can be
naturally regarded as a measure on $\cW^c(f)_v$ via the push forward under $p_2 \mid \cE_v$.

Let $\tSigma \subset \tcE$ be the ``half-closed'' set bounded by the diagonal section
of $\tcE$ and its image under $\tNG$, including the former and excluding the latter.
Notice that $\tSigma$ is $\tNF$-invariant and $\pi_1(S)$-invariant.
We denote by $\tmS$ the restriction of the measure $\tmE$ to $\tSigma$.
Then $\tmS$ is also $\tNF$-invariant and $\pi_1(S)$-invariant.

\begin{lemma}\label{l.projectionsagree}
$(\tp_1)_\ast\tmS =\tm = (\tp_2)_\ast\tmS$.
\end{lemma}

\begin{proof}
The first equality is a direct consequence of the normalization \eqref{eq.normalization}.
To prove the second one, begin by noting that
\begin{itemize}
\item [(i)] $(\tp_2)_\ast\tmS$ is the $\pi_1(S)$-invariant lift of a probability
measure on $M$.
\end{itemize}
Indeed $\tmS$ is $\pi_1(S)$-invariant, and if $\tSigma_0 \subset \tSigma$
is a fundamental domain for the $\pi_1(S)$-action on $\tcE$ then $\tmE(\tSigma_0)=1$,
since $(\tp_1)_\ast (\tmE \mid \tSigma)=\tm$. Moreover,
\begin{itemize}
\item [(ii)] $(\tp_2)_\ast\tmS$ is $\tf$-invariant.
\end{itemize}
That is because $\tmS$ is $\tNF$-invariant. Furthermore,
\begin{itemize}
\item [(iii)] $(\tp_2)_\ast\tmS$ is absolutely continuous with respect to $\tm$.
\end{itemize}
Indeed if a Borel set $X \subset \tM$ has zero $\tm$-measure, then $\tm_v(X)=0$
for $\tm$-almost every $v \in \tM$, and then the definition of $\tmE$ gives
$\tmE((\tp_2)^{-1}(X))=0$.
Since $f$ is $m$-ergodic, $\tm$ is the unique measure on $\tM$ satisfying properties (i)-(iii).
So, $(\tp_2)_\ast \tmS=\tm$ as claimed.
\end{proof}

Let $\Sigma$ be the projection of $\tSigma$ to $\cE$; equivalently, $\Sigma$ is the ``half-closed''
set bounded by the diagonal section of $\cE$ and its image under $\NG$, including the former and
excluding the latter. Let $\mS$ be the probability measure on $\cE$ induced by $\tmS$.
Note that $\mS$ gives zero measure to the complement of $\Sigma$, and hence
it is supported on the closure of $\Sigma$.
Moreover, $\mS$ is $\NF$-invariant
and satisfies
\begin{equation}\label{eq.projectingmeasures}
(p_1)_\ast \mS = m = (p_2)_\ast \mS.
\end{equation}
Recalling that almost every fiber $\cE_v$ is naturally identified with $\cW^c(f)_v$, we can write
$\mS =\left( m_v \mid [v,f(v))\right) \, dm(v)$.

\subsection{Lyapunov exponents}

Let $\chi^c(v)$ denote the center Lyapunov exponent of $f$ at a point $v\in M$,
that is
$$
\chi^c(v) = \lim_{n\to\infty} \frac{1}{n} \log \|Df^n \mid E^c_{v}\|.
$$
By ergodicity, there exists $\chi^c\in\RR$ such that $\chi^c(v)=\chi^c$ for
$m$-almost every $v\in M$. Since $E^c$ is $1$-dimensional, the ergodic
theorem ensures that $\chi^c$ can be expressed as an integral
$$
\chi^c = \int_{M} \log \|Df \mid E^c_v \| \, dm(v),
$$
with respect to any fixed Riemann structure on $M$.

\begin{lemma}\label{l.exponentscoincide}
The fiberwise exponent of the cocycle $\NF$ exists at a point $\xi\in \cE$ if and only
if the center Lyapunov exponent for $f$ exists at $p_2(\xi)$, and then the two are equal:
$$
\lim_{n\to \infty} \frac{1}{n} \log \|D_\xi \NF^n \|
= \chi^c(p_2(\xi)).
$$
In particular, the fiberwise exponent of the cocycle $\NF$ is equal to $\chi^c$ almost
everywhere with respect to $\mS$.
\end{lemma}

\begin{proof}
Note that the $\NF$-orbit of any $\xi \in \cE$ is precompact (indeed,
$\{\NG^k(\Sigma)\}_{k \in \ZZ}$ is a partition of $\cE$ into precompact $\NF$-invariant sets),
and so the existence and value of the fibered Lyapunov exponent at $\xi$ do not
depend on a particular choice of a fiberwise Riemannian metric.
Since $p_2$ restricts to an immersion on each fiber of $\cE$,
a particular choice of fiberwise Riemannian metric can be obtained by pulling
back the Riemannian metric on $M$ under $p_2$.  With respect to this metric,
we have the identity $\|D_{\xi} \NF\|=\|Df \mid E^c_{p_2(\xi)}\|$ (where the
derivative of $\NF$ is taken along the fibers of $\cE$). The conclusion follows.
\end{proof}

\section{The atomic case}

At this point there are two very different cases in our analysis: $\chi^c\neq 0$ and $\chi^c = 0$.
The first is handled easily by existing methods and implies that $\cW^c(f)$ has atomic
disintegration of volume. In handling the second case, we will introduce the meat of
the arguments in this paper.

\subsection{The case of nonvanishing center exponents}

Suppose that $\chi^c\neq 0$.  Let $X = \{v\in M \,\mid\, \chi^c(v) = \chi^c\}$,
which is a full measure subset of $M$.  Let $\cX = p_2^{-1}(X)\cap \Sigma$;
Lemma~\ref{l.exponentscoincide} implies that  $\cX$ is the set of $\xi\in \Sigma$
where the fiberwise exponent of $\NF$ is equal to $\chi^c$.
We want to use Theorem~\ref{t.rwmain} to conclude that $\cX$ coincides, up to zero
$\mS$-measure, with a measurable set $\cY\subset\Sigma$ meeting almost every fiber
of $\cE$ in finitely many points.

Strictly speaking, the theorem does not apply directly to the fiber bundle $\Sigma \to M$,
because its fibers are not compact.
However, this can be turned into a fiber bundle $\overline\Sigma \to M$ with compact fibers:
just take $\overline\Sigma$ to be the quotient of $\cE$ by $\NG$, so that the quotient map
restricts to a continuous bijection $P:\Sigma \to \overline\Sigma$.
The map $\NF$ goes down to the quotient to define a smooth cocycle $\overline \NF$ on
$\overline \Sigma$, which admits an invariant measure $P_\ast (\mS \mid \Sigma)$.
Fix an arbitrary Riemannian metric on the fibers of $\overline\Sigma$ depending continuously
on the base point; any two such metrics are uniformly equivalent, since $\overline\Sigma$
is compact. Notice that the restriction of $P$ to each fiber is smooth, with derivative
uniformly bounded away from zero and infinity, and so the fibered Lyapunov exponent of
$\overline\NF$ with respect to $P_\ast(\mS \mid \Sigma)$ is the same as the Lyapunov exponent
of $\NF$ with respect to $\mS\mid\Sigma$. Thus, we can apply Theorem~\ref{t.rwmain} in
$\overline\Sigma$, and then take the preimage under $P$ to obtain the conclusion in $\Sigma$.

By construction, the family $\{m_v \mid [v, f(v)) : v\in M\}$ is a disintegration of $\mS$
along $\cE$ fibers. Since $\cY$ has full $\mS$-measure, its intersection with almost every fiber
has full conditional measure on the fiber. This implies that $m_v \mid [v,f(v))$ is atomic,
with finitely many atoms, for $m$-almost every $v$. The function that assigns to each $v\in M$
the number of atoms is a measurable, $f$-invariant function. So, ergodicity of $f$ implies that
this number is $m$-almost everywhere constant. Let $k\ge 1$ be this constant.
Then there exists some full $\mS$-measure set $\cZ\subset \Sigma$ whose intersection with
almost every fiber $\cE_v$ coincides with the support of $m_v \mid [v,f(v))$ and contains
exactly $k$ points.

The projection $p_2(\cZ)$ is a full $m$-measure subset of $M$, by property \eqref{eq.projectingmeasures}.
Moreover, $p_2(\cZ)$ is $f$-invariant, because $m_v$ is $f$-invariant;
recall \eqref{eq.normalization5} and \eqref{eq.normalization6}. Since $[v,f(v))$ is a fundamental
domain for the action of $f$ on any noncompact center leaf, it follows that the intersection of
$p_2(\cZ)$ with almost every $\cW^c(f)_v$ consists of exactly $k$ orbits, whose points are the
atoms of the corresponding measure $m_v$. Then $p_2(\cZ)$ coincides, up to zero $m$-measure,
with some measurable set that intersects every center leaf in exactly $k$ orbits.
So, alternative (1) of Theorem A holds in the case where $\chi^c\neq 0$.

\subsection{Vanishing center exponents: using the invariance principle}

Now let us suppose that $\chi^c= 0$. Using the invariance principle stated in
Theorem~\ref{t.asvmain}, we prove:

\begin{lemma}\label{l.jimmy}
There is a continuous\footnote{We recall that the space of Radon measures on
$\cE$ can be seen as a cone in the dual of the space of
compactly supported continuous functions on $\cE$, and
hence inherits a natural weak-$*$ topology.}
family $\{\hm_v: v\in M\}$ of Radon measures
on the fibers of $\cE$ with the following properties:
\begin{enumerate}
\item $\hm_v = m_v$  for $m$-almost every $v\in M$;
\item the family is $\rho$-invariant; in particular,
$$
\NF_\ast \hm_v = \hm_{f(v)} \quand \NG_\ast \hm_v = \hm_{v}\quad\text{for all $v\in M$;}
$$
\item the family is invariant under $su$-holonomy:
$$
(H^s_{v,v'})_\ast \hm_v = \mSc_{v'} \quand (H^u_{w,w'})_\ast\mSc_w = \mSc_{w'}
$$
for all $v'\in \cW^s(f)_v$ and $w'\in \cW^u(f)_w$.
\end{enumerate}
\end{lemma}

\begin{proof}
Note that $f$ satisfies the hypotheses of Theorem~\ref{t.asvmain}: it is partially hyperbolic,
volume-preserving, center bunched (since $E^c$ is $1$-dimensional) and accessible
(since $\varphi_1$ is stably accessible). As we have seen, the bundle $\cE$ admits $su$-holonomy,
and the probability measure $\mS$ on $\cE$ projects to the volume $m$ and is invariant under
the smooth cocycle $\NF$. We are in the case where $\chi^c= 0$, which by Lemma~\ref{l.exponentscoincide}
implies that the fiberwise Lyapunov exponent for $\NF$ vanishes $\mS$-almost everywhere.
Applying Theorem~\ref{t.asvmain}, we conclude that there is a continuous $\NF$-invariant
and $su$-holonomy invariant family of probability measures supported on the fibers of $\Sigma$
and agreeing $m$-almost everywhere with the disintegration $\{m_v \mid [v, f(v))\}$ of $\mS$.
Since $[v, f(v))$ is a fundamental domain for the action of $\NG$ on the fiber $\cE_v$,
we can extend this continuous family of probabilities to a continuous family of $\sigma$-finite
measures $\hm_v$ supported on the fibers of $\cE$. By construction, this family is
invariant under $\rho$ and under $su$-holonomy. Moreover, it agrees $m$-almost everywhere
with the family $\{ m_v\}$.
\end{proof}

The family of measures  $\{\hm_v: v\in M\}$  given by Lemma~\ref{l.jimmy} is a disintegration
of $\mE$ and shares some properties with the family $\{ m_v : v\in M\}$,
for example $\rho$-invariance. The family $\{\hm_v : v\in M\}$ has the extra properties of
continuity and invariance under $su$-holonomy.
On the other hand, the family  $\{ m_v : v\in M\}$ has one extra property that is
not \emph{a priori} enjoyed by $\{\hm_v: v\in M\}$: invariance under $c$-holonomy.
This reflects the fact that $\{ m_v : v\in M\}$ comes from a disintegration of $m$ along
local $\cW^c(f)$-leaves, and is not just an arbitrary disintegration of $\mE$ along $\cE$ fibers.

We can characterize whether $\{\hm_v: v\in M\}$ is invariant under $c$-holonomy by
looking at the supports $\supp \hm_v$ of these measures on the fibers.
To this end we show:

\begin{lemma}\label{l.suppdichotomy} Either
\begin{itemize}
\item[(i)] there exists $k\in \NN$ such that $\#\,\supp \hm_v \cap\Sigma = k$ for all $v\in M$
\item[(ii)] or $\supp \hm_v = \cE_v$ for all $v\in M$.
\end{itemize}
\end{lemma}

The key ingredient in the proof of Lemma~\ref{l.suppdichotomy} is the following lemma, which shows that the measures $\hm_v$ have a strong homogeneity property under holonomy maps.

\begin{lemma}\label{l.homogeneity}  For any $v\in M$ and  for any $\xi,\xi'\in \supp\hm_v$,
there is an orientation-preserving $C^1$
diffeomorphism $H_{\xi,\xi'}\colon \cE_v\to \cE_v$ (a composition of $s,u$ and $c$ holonomies in $\cE$) with the following properties:
\begin{enumerate}
\item $H_{\xi,\xi'}(\xi) = \xi' $;
\item $(H_{\xi,\xi'})_\ast \hm_v = \hm_v$;
\item if $\xi,\xi'\in \supp\hm_v$, then $H_{\xi,\xi'}(\supp \hm_v) = \supp \hm_v$;
\item if $f$ is $r$-bunched, then $H_{\xi,\xi'}$ is a $C^r$ diffeomorphism.
\end{enumerate}
\end{lemma}

\begin{proof}[Proof of Lemma~\ref{l.homogeneity}]
Note that $\hm_w = m_v$ for every $w \in \supp m_v$ and almost every $v$, because $\hm_v = m_v$
almost everywhere, $\hm_v$ is continuous in $v$, and $m_v$ is constant on every center leaf.

Let $w, w'$ be the $p_2$-projections of $\xi,\xi'$. By accessibility of $f$, there is an $su$-path $\gamma$ in $M$ connecting $w$ to $w'$.
Since $p_1$ maps leaves of $\cF^*$ homeomorphically to leaves of $\cW^*(f)$, for $*\in\{s, u\}$,
we can lift $\gamma$ to an $su$-path in $\cE$ connecting $\eta=(w,w)$ to $\eta'=(w',w')$.
Let $H:\cE_w \to \cE_{w'}$ be the $su$-holonomy map along this $su$-path. Then $H$ sends
$\eta$
to $\eta'$ and, since the disintegration $\{\hm_u: u\in M\}$ is invariant under $su$-holonomy,
it maps $\hm_w$ to $\hm_{w'}$.

Suppose first that $v\in\supp m_v$ (this holds
$m$-almost everywhere). Then the condition $\xi\in\supp\hm_v$ means that $w\in\supp m_v$,
which implies $\hm_w = m_v = \hm_v$. Analogously, $w'\in\supp m_v$ and $\hm_{w'} = m_v = \hm_v$.  Identifying the fibers $\cE_w, \cE_{w'}$ to $\cE_v$ through
$c$-holonomy in $\cE$, we obtain a homeomorphism $H_{\xi,\xi'}:\cE_v \to \cE_v$ satisfying properties (1)-(3).

The assumption on $v$ is readily removed, as follows. Given any $v\in M$ let $v_0$ be any point
such that $v_0 \in \supp m_{v_0}$ and let $\gamma$ be an $su$-path in $M$ connecting $v$ to
$v_0$. The $su$-holonomy $H_0:\cE_v \to \cE_{v_0}$ along the lift of $\gamma$ maps
$\supp \hm_v$ to $\supp \hm_{v_0}$. Let $\xi_0, \xi_0'$ be the
images of $\xi, \xi'$ under $H_0$.
Conjugating $H_{\xi_0,\xi_0'}$ by $H_0$ we obtain a homeomorphism
$H_{\xi,\xi'}$ satisfying conclusions (1)-(3).

Since $\tf$ is partially hyperbolic with $1$-dimensional center it is center bunched, and so the (globally defined) $su$-holonomy maps between $\cW^c(\tf)$ leaves are $C^1$.  This implies that $H_{\xi,\xi'}$ is a $C^1$ diffeomorphism.  Moreover, if $f$ is $r$-bunched, then so is $\tf$, and the
leaves of $\cW^c(\tf)$ and all holonomies are $C^r$; in this case  $H_{\xi,\xi'}$ is a $C^r$ diffeomorphism, verifying property (4).
\end{proof}

\begin{proof}[Proof of Lemma~\ref{l.suppdichotomy}]
The support of each $\hm_v$ is a locally
compact subset of the fiber $\cE_v$. If $\supp \hm_v$ has bilateral accumulation
points for some (and hence all) $v\in M$, then Lemma~\ref{l.homogeneity} implies that $\supp \hm_v = \cE_v$: otherwise one would have an interval in the complement of $\supp \hm_v$ whose boundary points
would fail to be bilateral accumulation points. This means that conclusion (ii) holds.

If $\supp \hm_v$ has no bilateral accumulation
points then it is countable; since it is locally compact, it therefore
contains (and hence consists of) isolated points.
Hence, the support of every $\hm_v \mid [v,f(v))$ is finite.
But $\mS = \left(\hm_v \mid [v,f(v))\right) \, dm(v)$ is $\NF$-invariant, and so
$\# \, \supp(\hm_v \mid [v,f(v)) ) $ is an $f$-invariant positive measurable function.
By ergodicity of $f$, there exists $k\ge 1$ such that $\#\,\supp(\hm_v \mid [v,f(v)) ) = k$
for $m$-almost all $v$. Conclusion (i) follows, using the continuity of $\hm_v$.
\end{proof}

We call alternative (i) of Lemma~\ref{l.suppdichotomy} the atomic case, and alternative (ii)
the continuous case. Let us consider the atomic case first.
Then, for every $v$ in some full $m$-measure subset,
$\supp m_v\mid [v,f(v)) = \supp \hm_v \mid[v,f(v))$ consists of exactly $k$ points.
Since $[v,f(v))$ is a fundamental domain for the action of $f$ on $\cW^c(f)_v$, assuming
the center leaf is non-compact, it follows that the support of $m_v$ consists of exactly
$k$ orbits, for every $v$ in some full $m$-measure Borel set $M_0\subset M$.  We
may further assume that $\hm_v=m_v$ for $v \in M_0$.  Taking the unions
of the supports
$$
\bigcup_{v\in M_0} \supp m_v = p_2\left( \bigcup_{v\in M_0} \supp \hm_v\right)
$$
we obtain a full measure set\footnote{Note that $\bigcup_{v \in M_0} \supp
\hm_v$ is a Borel subset (by continuity of $\hm_v$), hence its image under
$p_2$ is analytic, and hence Lebesgue measurable.} meeting almost every $\cW^c(f)$-leaf in exactly $k$ orbits of $f$.
Then there exists a full measure set such that this happens for every center leaf, as claimed
in alternative (1) of Theorem A.

Finally, observe that in the atomic case the family $\{\hm_v: v\in M\}$ is not invariant under
$c$-holonomy. Indeed, consider any $v$ such that $v\in \supp m_v$ and let
$w\in\cW^c(f)_v \setminus \supp m_v$. Then $v\in\supp\hm_v$ and, using accessibility,
$w\in \supp\hm_w$. The latter implies that $\hm_w \neq m_v = \hm_v$.

\section{The continuous case}

One case remains in the proof of Theorem A, in which the fiberwise exponent $\chi^c$ vanishes
and the family of measures $\{\hm_v: v\in M\}$ has full support on the fibers of $\cE$.
We shall see that in this case the family $\{\hm_v: v\in M\}$ is invariant under
$c$-holonomy and can be used to define a continuous disintegration of volume along center leaves.
The existence of this disintegration leads to alternative (2) in Theorem A:
the existence of a smooth vector field in which $f$ embeds. Then the center foliation is
smooth and, in particular, has Lebesgue disintegration.  The first step is to establish $c$-invariance of the measures $\{\hm_v : v\in M\}$.

\begin{lemma}
The family $\{\hm_v : v\in M\}$ is invariant under $c$-holonomy of $\NF:\cE\to\cE$.
\end{lemma}

\begin{proof}
The fact that $\supp \hm_v = \cE_v$ for every $v$ implies that every set of
full $\hm$-measure must be dense in almost every fiber.
Recall that $\hm_v = m_v$ almost everywhere and
$\{m_v : v\in M\}$ is invariant under $c$-holonomy of $\cE$. This implies that
$\{\hm_v : v\in M\}$ is invariant under $c$-holonomy restricted to a dense set of
points in a dense set of center leaves. Since the family  $\{\hm_v : v\in M\}$ is
continuous, it follows that is invariant under $c$-holonomy on the whole of $\cE$.
\end{proof}

\subsection{Absolute continuity of $\cW^c(f)$}

For $v\in M$, denote by $\lambda_v$ the Riemannian measure on the fiber
and denote by $I(\xi,r)$  the interval in $\cE_v$ centered at $\xi$ of radius $r$,
with respect to the $p_2$-pullback metric of the Riemann structure on $\cW^c(\tf)_v$.

\begin{lemma}\label{l.ac1} For each $v\in M$, the measure $\hm_v$ is equivalent to
Lebesgue measure $\lambda_v$.  The limit
$$
\Delta_v(\xi) =  \lim_{r\to 0}\frac{\hm_v(I(\xi,r))}{\lambda_v(I(\xi,r))}
$$
exists everywhere, is continuous, and takes values in $(0,\infty)$.
\end{lemma}

\begin{proof}  For $v\in M$ and  $\xi\in \cE_v$ let
$$
\overline\Delta_v(\xi) = \limsup_{r\to 0}\frac{\hm_v(I(\xi,r))}{\lambda_v(I(\xi,r))},\qquad
\underline\Delta_v(\xi) = \liminf_{r\to 0}\frac{\hm_v(I(\xi,r))}{\lambda_v(I(\xi,r))}.
$$
For $\hm_v$-almost every $\xi\in \cE_v$, we have
$$
\overline\Delta_v(\xi)=\underline\Delta_v(\xi) \in (0,\infty].
$$
Since $\supp\hm_v = \cE_v$,  Lemma~\ref{l.homogeneity} implies that for any two points
$\xi, \xi' \in \cE_v$, there is a diffeomorphism $H_{\xi,\xi'}\colon \cE_v\to \cE_v$ preserving
$\hm_v$ and sending $\xi$ to $\xi'$. Since $C^1$ diffeomorphisms have continuous and positive Jacobians,
it follows that for any $\xi,\xi'$:
$$
\underline\Delta_v(\xi) = \overline\Delta_v(\xi)  \quad\iff\quad  \underline\Delta_v(\xi') = \overline\Delta_v(\xi').
$$
Thus $ \underline\Delta_v= \overline\Delta_v$ everywhere on $\cE_v$; denote this function by $\Delta_v$.

Then $\hm_v$ has a singular part with respect to $\lambda_v$ if and only if there is a positive
$\hm_v$-measure set $X \subset \cE_v$ such that, for $\xi\in X$, $\Delta_v(\xi) = \infty $. On the other hand,
again using the diffeomorphisms $H_{\xi,\xi'}$ we see that for every $\xi,\xi'$:
$$
\Delta_v(\xi) = \infty \quad\iff\quad  \Delta_v(\xi') = \infty .
$$
Hence if $\hm_v$ had a singular part with respect to $\lambda_v$, this would imply that $\Delta_v\equiv\infty$ on $\cE_v$,
contradicting the local finiteness of $\hm_v$.  Therefore $\hm_v$ is absolutely continuous with respect to $\lambda_v$.
Similarly, we see that $\lambda_v$ is absolutely continuous with respect to $\hm_v$, and so the two measures are equivalent.

The function $\Delta$ is a pointwise limit of the continuous functions
$$
\xi\mapsto \frac{\hm_v(I(\xi,r))}{\lambda_v(I(\xi,r))}
$$
and hence is a Baire class 1 function;  it follows that $\Delta$ has a point of continuity \cite[Theorem 7.3]{Ox80}.
Again using Lemma~\ref{l.homogeneity}, we see that every point in $\cE$ is a point of  continuity
of $\Delta$, and so $\Delta$ is continuous.
\end{proof}

Recall that for almost every $v\in M$, we have  $\hm_v = m_v$, where $m_v$ is a representative  of the disintegration
of volume on the (noncompact) leaf $\cW^c(f)_v$.  The previous lemma thus implies  that $m_v$ is equivalent to
Lebesgue measure on $\cW^c(f)_v$, for almost every $v$. We conclude:
\begin{lemma}
$\cW^c(f)$  is leafwise absolutely continuous.
\end{lemma}

\subsection{Embedding $\tf$ in a continuous flow}

Consider the continuous vector field  $Z$ on $\cE$ given by
$$
Z(\xi) =  \frac{Z_0(\xi)}{\Delta(\xi)},
$$
where $Z_0$ is the positively oriented unit speed vector field tangent to the fibers of $\cE$
(with respect to the $p_2$-pullback metric).
Since $\Delta_v = d\hm_v /d\lambda_v$, it follows that $Z$
generates a flow $\phi_t$ on $\cE$  satisfying
$$
\hm_v\left([\xi, \phi_t(\xi))\right) = t,
$$
for all  $v\in M$, $\xi\in \cE_v$ and $t\in\RR$.    Since
$\hm_v[\xi, \NG(\xi)) = 1$, it follows  that $\phi_1(\xi) = \NG(\xi)$, for all $\xi$.

The invariance properties of  $\hm_v$ translate into invariance properties of the flow:
\begin{itemize}
\item $\phi_t$ commutes with the $\rho$-action on $\cE$;
\item $\phi_t$ commutes with $u, s$ and $c$ holonomy.
\end{itemize}
The analogous properties holds for the vector field $Z$; in particular:
\begin{itemize}
\item $Z$ is preserved by the $\rho$-action on $\cE$:
\item $Z$ is preserved by  $u, s$ and $c$ holonomy.
\end{itemize}
The $c$-invariance of $Z$ implies that $Z$ projects under $p_2$ to a well-defined continuous vector field $X$
on $M$ tangent to the leaves of $\cW^c(f)$.  The $\NG$-invariance of $Z$ implies that $f_\ast X= X$.
Let $\psi_t$ be the flow generated by $X$; it satisfies $\phi_t\circ p_2
= p_2\circ \psi_t$ for all $t$.  Since $\phi_1 = \NG$ and $p_2\circ \NG = f \circ p_2$,
we have that $\psi_1 = f$; in other words, $f$ embeds in the flow $\psi_t$.

\begin{lemma} The flow $\psi$ preserves the volume $m$.
\end{lemma}

\begin{proof} Fix $t\in \RR$.  Since $\cW^c(f)$ is leafwise absolutely continuous and $\psi_t$ is $C^1$
along the leaves of $\cW^c(f)$, the map $\psi_t$ preserves the measure class of $m$.
Hence $\psi_t$ has a Jacobian with respect to volume:
$$
\Jac(\psi_t) = \frac{d\left((\psi_t)^\ast m\right)}{dm}.
$$
Since $\psi_t \circ f = f\circ \psi_t$, it follows that $\Jac\psi_t(f(t)) = \Jac(\psi_t)$.
Ergodicity of $f$ implies that $\Jac(\psi_t)$ is almost everywhere constant, and hence almost everywhere equal to $1$.
This immediately implies that $(\psi_t)_\ast m =m$.
\end{proof}

\subsection{Showing that the flow is smooth}

Let $\psi$ be the volume-preserving flow on $M$ satisfying
$f = \psi_1$ that we have just constructed. Our remaining task is to prove that the flow $\psi$ is $C^\infty$.
This is accomplished in two steps.  In the first step, we show that $\psi_t$ is $C^\infty$ along the leaves of $\cW^c(f)$.
In the second step we show that $\psi_t$ is $C^\infty$ along the leaves of $\cW^u(f)$ and $\cW^s(f)$.
A straightforward application of a result of Journ\'e then shows that the flow $\psi$ is $C^\infty$.

To show $C^\infty$ smoothness along the leaves of $\cW^c(f)$ one first must establish that the leaves of $\cW^c(f)$ are $C^\infty$.
A priori, these leaves have only finite smoothness determined by the $C^1$ distance from $f$ to $\varphi_1$.
However in the case under consideration, in which volume has Lebesgue disintegration along $\cW^c(f)$ leaves,
we have more information about the action of $f$ on center leaves.

In particular, since $f$ preserves a continuous vector field $X$ tangent to $\cW^c(f)$-leaves, the derivatives $(Df^k \mid E^c)$,
$k\in\ZZ$ of its iterates along the central
direction are uniformly bounded.
This implies that $f$ is $r$-bunched for every $r>0$; recall \eqref{eq.rbunched}.
Hence, the leaves of $\cW^{cs}(f)$, $\cW^{cu}(f)$ and $\cW^c(f)$ are $C^\infty$,
and the $\cW^s(f)$-holonomies and $\cW^u(f)$-holonomies between $\cW^c(f)$-leaves
are also $C^\infty$.

This in turn implies that the fibers of $\cE$ and the  diffeomorphisms $H_{\xi,\xi'}$ given by
Lemma~\ref{l.homogeneity} are  $C^\infty$.   We will use this information  to conclude that the function $\Delta$ is $C^\infty$
along the fibers of $\cE$, which implies that $X$ is $C^\infty$ along the leaves of $\cW^c(f)$.

\begin{lemma}\label{l.centersmooth} The function $\Delta$ given by Lemma~\ref{l.ac1} is $C^\infty$ along the fibers of
$\cE$, with derivatives varying continuously from fiber to fiber.  Consequently $X$ is $C^\infty$ along the leaves of $\cW^c(f)$,
uniformly in the leaves, as is the flow $\psi_t$.
\end{lemma}

\begin{proof}
Fix $v\in M$.  For any $\xi\in\cE_v$ and any diffeomorphism $H$ of $\cE_v$ preserving $\hm_v$, we have
\begin{equation}\label{e.deltapres}
\Delta_v(H(\xi)) =   \frac{\Delta(\xi)}{\Jac(H)(\xi)} .
\end{equation}
If $H$ is $C^\infty$, then so is the Jacobian $\Jac(H)$.
Consider the graph of $\Delta_v$:
$$\graph(\Delta_v) = \{ (\xi, \Delta(\xi)) :  \xi \in \cE_v\}\subset \cE_v\times \RR.
$$
Since the function $\Delta$ is continuous, $\graph(\Delta_v)$ is locally compact.
If $H$ is an $\hm_v$-preserving $C^\infty$ diffeomorphism, then (\ref{e.deltapres}) implies that the $C^\infty$ diffeomorphism
$$
(\xi, t) \mapsto (H(\xi), \frac{t}{\Jac(H)(\xi) })
$$
preserves $\graph(\Delta_v)$.

Combining this observation with Lemma~\ref{l.homogeneity}, we obtain that  for any pair of points
$q=(\xi,\Delta_v(\xi))$ and $q' =(\xi',\Delta_v(\xi'))$ in  $\graph(\Delta_v)$, there is a $C^\infty$ diffeomorphism
of $\cE_v\times \RR$ sending $q$ to $q'$ and preserving $\graph(\Delta_v)$.
That is, the locally compact set  $\graph(\Delta_v)$ is $C^\infty$ homogeneous.  A result of
Repov{\v{s}}, Skopenkov and  {\v{S}}{\v{c}}epin \cite{RSS96} implies that $\graph(\Delta_v)$ is a $C^\infty$ submanifold of
$\cE_v\times \RR$ (see also \cite{Wliv}).  Thus $\Delta_v$ is $C^\infty$ off of its  singularities (by ``singularities,"
we mean points where the projection of $\graph(\Delta_v)$ onto $\cE_v$ fails to be a submersion).
But if $\Delta_v$ has any singularities, then it is easy to see that {\em  every point} in $\cE_v$ must be a singularity,
which violates Sard's theorem. Hence $\Delta_v$ has no singularities and therefore is $C^\infty$.

To see that the derivatives of $\Delta_v$ vary continuously as a function of $v$, note that one can
move from the fiber over $v$ to any neighboring fiber by a composition of local $u$, $s$ and $c$-holonomies.
The derivatives of these holonomy maps very continuously with the fiber.  Equation
(\ref{e.deltapres}) implies that the fiberwise derivatives vary continuously.
 \end{proof}

Our next step is to establish the $C^\infty$ smoothness of $X$ along $\cW^s(f)$ and $\cW^u(f)$ leaves.
Note first that because $\psi_t\circ f = f \circ \psi_t$ for all $t$, it follows that
$\psi_t$ preserves the foliations $\cW^s(f)$ and $\cW^u(f)$. To see this, observe that, since $f$ preserves
a nonvanishing continuous vector field, the leaf of $\cW^s(f)$ through $v$ is uniquely characterized
as the set of points $w$ such that
$$
\lim_{n\to\infty} d(f^n(v), f^n(w)) = 0;
$$
since $\psi_t$ commutes with $f$, then for such $v,w$, we will also have
$$
\lim_{n\to\infty} d(f^n(\psi_t(v)), f^n(\psi_t(w))) =  \lim_{n\to\infty} d(\psi_t(f^n(v)), \psi_t(f^n(w)))  = 0;
$$
and so $\psi_t(w) \in \cW^s(f)_{\psi_t(v)}$.

We first show that for every $t$ the restriction of $\psi_t$ to the leaves of $\cW^s(f)$
is uniformly $C^\infty$.  Here we use the property that $\psi_t$ preserves volume.
The basic idea is that $\psi_t$ must also preserves the disintegration of volume
along $\cW^s(f)$ leaves in a foliation box, up to a constant scaling factor along each leaf.
But these disintegrations are $C^\infty$ along $\cW^s(f)$ leaves;  when the leaves are one-dimensional,
this forces $\psi_t$ to be $C^\infty$ along the leaves as well.

The following lemma is well-known (see formula (11.4) in \cite{Beyond}):
\begin{lemma}\label{l.smooth1} Let $f\colon M\to M$ be any $C^\infty$
partially hyperbolic diffeomorphism.
For any foliation box $\cB\subset M$ for $\cW^s(f)$, there is a continuous disintegration
of $m \mid\cB$ along leaves of $\cW^s(f)$ (defined at every point $p\in \cB$).
These disintegrations are equivalent to Riemannian measure in the $\cW^{s}(f)$ leaves.
The densities of the disintegrations are $C^\infty$ along leaves and transversely
continuous. The same is true for $\cW^{u}(f)$.
\end{lemma}

\begin{lemma}\label{l.smooth2}
For any $\cW^s(f)$ foliation box $\cB$, any $t\in\RR$, and any $v\in \cB$,
the map $\psi_t$ sends the disintegration $m^{s}_v$
of $m \mid \cB$ along $\cW^{s}(f)$ leaves at $v$ to the disintegration
$m^{s}_{\psi_t(v)}$ of $m\vert_{\psi_t(\cB)}$ along $\cW^{s}(f)$
leaves at $\psi_t(v)$. The same is true for $\cW^{u}(f)$.
\end{lemma}

\begin{proof}
Denote by $\{m^{s}_v: v\in \cB\}$ the disintegration of $m$ along
$\cW^{s}(f)$ leaves inside the box $\cB$. By Lemma~\ref{l.smooth1},
the map $v\mapsto m^{s}_v$ is continuous.

Fix $t\in \RR$.  Since $\psi_t$ preserves both $m$ and the
leaves of $\cW^s(f)$ , we obtain that
\begin{eqnarray}\label{e.preservenu}
{\psi_t}_\ast m^{s}_v &=& m^{s}_{\psi_t(v)},
\end{eqnarray}
for $m$-almost every $v\in \cB$, where the disintegration on the
right hand side takes place in the box $\psi_t(\cB)$.
 Since $v\mapsto m^{s}_v$ is continuous (on both sides of the equation)
and $\psi_t$ is a homeomorphism, equation \eqref{e.preservenu}
holds everywhere.
\end{proof}

\begin{lemma}\label{l.smooth3}
For every $t\in\psi_t$, the map $\psi_t$ is uniformly $C^\infty$
along $\cW^{s}(f)$ leaves and uniformly $C^\infty$ along $\cW^{u}(f)$ leaves.
\end{lemma}

\begin{proof}
Lemma~\ref{l.smooth2} implies that $\psi_t$ satisfies an ordinary differential
equation along $\cW^{s}(f)$ leaves with $C^\infty$ (and transversely
continuous) coefficients, and so the solutions are
$C^\infty$ and vary continuously
with the leaf.
\end{proof}

At this point, we have shown that for every $t\in \RR$ and $v\in M$:
\begin{enumerate}
\item the restriction of $\psi_t$ to $\cW^c(f)_v$ is $C^\infty$ (Lemma~\ref{l.centersmooth} );
\item the restriction of $\psi_t$ to $\cW^s(f)_v$ and $\cW^u(f)_v$ are $C^\infty$ (Lemma~\ref{l.smooth3});
\item the map $t\mapsto \psi_t(v)$ is $C^\infty$  (Lemma~\ref{l.centersmooth}).
\end{enumerate}
Moreover, these statements hold uniformly in $t$ and $v$.

Our final tool is the so-called ``Journ\'e Lemma" which allows us to deduce smoothness of a
function by checking along leaves of two transverse foliations with smooth leaves:

\begin{theorem}[Journ{\'e} \cite{Jo88}]\label{t.journe}
Let $\cF_1$ and $\cF_2$ be transverse foliations of a manifold $M$
whose leaves are uniformly $C^\infty$. Let $\psi: M\to \RR$ be any
continuous function such that the restriction of $\psi$ to the
leaves of $\cF_1$ is uniformly $C^{\infty}$ and  the restriction of
$\psi$ to the leaves of $\cF_2$ is uniformly $C^{\infty}$.  Then $\psi$
is uniformly $C^{\infty}$.
\end{theorem}

Let us use this result to finish the proof of Theorem A.  To show that the flow $\psi$ is
$C^\infty$, we must show that it is $C^\infty$ on $M\times \RR$.  Fix $t$ and $v$,
and consider the restriction of $\psi_t$ to the leaf $\cW^{cs}(f)_v$. This leaf is uniformly
subfoliated by the foliations $\cW^{c}(f)$ and $\cW^s(f)$.  The map $\psi_t$ is uniformly
$C^\infty$ when restricted to the leaves of each of these foliations.
Theorem~\ref{t.journe} implies that $\psi_t$ is uniformly $C^\infty$ along leaves of $\cW^{cs}(f)$.
But the restriction of $\psi_t$ to leaves of $\cW^{u}(f)$ is also  uniformly $C^\infty$.
Applying Theorem~\ref{t.journe} again, we obtain that for each $t$, the map $\psi_t$ is $C^\infty$ on $M$,
uniformly in $t$.  Since $t\mapsto \psi_t(v)$ is $C^\infty$, uniformly in $v$, a final application of
Theorem~\ref{t.journe} gives that $\psi$ is $C^\infty$ on $M\times \RR$.

\bibliographystyle{plain}
\bibliography{bib}

\begin{thebibliography}{10}

\bibitem{AbV}
F.~Abdenur and M.~Viana.
\newblock Flavors of partial hyperbolicity.
\newblock Preprint www.impa.br/$\sim$viana/2009.

\bibitem{An67}
D.~V. Anosov.
\newblock Geodesic flows on closed {R}iemannian manifolds of negative
  curvature.
\newblock {\em Proc. Steklov Math. Inst.}, 90:1--235, 1967.

\bibitem{AS67}
D.~V. Anosov and Ya.~G. Sinai.
\newblock Certain smooth ergodic systems.
\newblock {\em Russian Math. Surveys}, 22:103--167, 1967.

\bibitem{ASV}
A.~Avila, J.~Santamaria, and M.~Viana.
\newblock Cocycles over partially hyperbolic maps.
\newblock Preprint www.preprint.impa.br 2008.

\bibitem{AV3}
A.~Avila and M.~Viana.
\newblock Extremal {L}yapunov exponents: an invariance principle and
  applications.
\newblock {\em Inventiones Math.}, 181:115--178, 2010.

\bibitem{BPS99}
L.~Barreira, Ya. Pesin, and J.~Schmeling.
\newblock Dimension and product structure of hyperbolic measures.
\newblock {\em Ann. of Math.}, 149:755--783, 1999.

\bibitem{BFL92}
Y.~Benoist, P.~Foulon, and F.~Labourie.
\newblock Flots d'{A}nosov \`a distributions stable et instable
  diff\'erentiables.
\newblock {\em J. Amer. Math. Soc.}, 5:33--74, 1992.

\bibitem{BL93}
Y.~Benoist and F.~Labourie.
\newblock Sur les diff\'eomorphismes d'{A}nosov affines \`a feuilletages stable
  et instable diff\'erentiables.
\newblock {\em Invent. Math.}, 111:285--308, 1993.

\bibitem{Beyond}
C.~Bonatti, L.~J. D{\'{\i}}az, and M.~Viana.
\newblock {\em Dynamics beyond uniform hyperbolicity}, volume 102 of {\em
  Encyclopaedia of Mathematical Sciences}.
\newblock Springer-Verlag, 2005.

\bibitem{BGV03}
C.~Bonatti, X.~G{\'o}mez-Mont, and M.~Viana.
\newblock G\'en\'ericit\'e d'exposants de {L}yapunov non-nuls pour des produits
  d\'eterministes de matrices.
\newblock {\em Ann. Inst. H. Poincar\'e Anal. Non Lin\'eaire}, 20:579--624,
  2003.

\bibitem{BoV04}
C.~Bonatti and M.~Viana.
\newblock Lyapunov exponents with multiplicity 1 for deterministic products of
  matrices.
\newblock {\em Ergod. Th. {\&} Dynam. Sys}, 24:1295--1330, 2004.

\bibitem{BP74}
M.~Brin and Ya. Pesin.
\newblock Partially hyperbolic dynamical systems.
\newblock {\em Izv. Acad. Nauk. SSSR}, 1:177--212, 1974.

\bibitem{BS02}
M.~Brin and G.~Stuck.
\newblock {\em Introduction to dynamical systems}.
\newblock Cambridge University Press, 2002.

\bibitem{BHHTU}
K.~Burns, F.~Rodriguez Hertz, M~A.~Rodriguez Hertz, A.~Talitskaya, and R.~Ures.
\newblock Density of accessibility for partially hyperbolic diffeomorphisms
  with one-dimensional center.
\newblock {\em Discrete Contin. Dyn. Syst.}, 22:75--88, 2008.

\bibitem{BPSW}
K.~Burns, C.~Pugh, M.~Shub, and A.~Wilkinson.
\newblock Recent results about stable ergodicity.
\newblock In {\em Smooth ergodic theory and its applications (Seattle WA,
  1999)}, volume~69 of {\em Procs. Symp. Pure Math.}, pages 327--366. Amer.
  Math. Soc., 2001.

\bibitem{BPW00}
K.~Burns, C.~Pugh, and A.~Wilkinson.
\newblock Stable ergodicity and {A}nosov flows.
\newblock {\em Topology}, 39:149--159, 2000.

\bibitem{BW08b}
K.~Burns and A.~Wilkinson.
\newblock Dynamical coherence and center bunching.
\newblock {\em Discrete Contin. Dyn. Syst.}, 22:89--100, 2008.

\bibitem{BW10}
K.~Burns and A.~Wilkinson.
\newblock On the ergodicity of partially hyperbolic systems.
\newblock {\em Annals of Math.}, 171:451--489, 2010.

\bibitem{Di03}
P.~Didier.
\newblock Stability of accessibility.
\newblock {\em Ergod. Th. {\&} Dyanm. Sys.}, 23:1717--1731, 2003.

\bibitem{Dol04}
D.~Dolgopyat.
\newblock On differentiability of {SRB} states for partially hyperbolic
  systems.
\newblock {\em Invent. Math.}, 155:389--449, 2004.

\bibitem{DW03}
D.~Dolgopyat and A.~Wilkinson.
\newblock Stable accessibility is {$C^1$} dense.
\newblock {\em Ast\'erisque}, 287:33--60, 2003.

\bibitem{EKL06}
M.~Einsiedler, A.~Katok, and E.~Lindenstrauss.
\newblock Invariant measures and the set of exceptions to {L}ittlewood's
  conjecture.
\newblock {\em Ann. of Math.}, 164:513--560, 2006.

\bibitem{Fu63}
H.~Furstenberg.
\newblock Non-commuting random products.
\newblock {\em Trans. Amer. Math. Soc.}, 108:377--428, 1963.

\bibitem{FK60}
H.~Furstenberg and H.~Kesten.
\newblock Products of random matrices.
\newblock {\em Ann. Math. Statist.}, 31:457--469, 1960.

\bibitem{HHTU10b}
F.~Rodriguez Hertz, M.~A.~Rodriguez Hertz, A.~Tahzibi, and R.~Ures.
\newblock Maximizing measures for partially hyperbolic systems with compact
  center leaves.
\newblock Preprint 2010.

\bibitem{HHUcoh}
F.~Rodriguez Hertz, M.~A.~Rodriguez Hertz, and R.~Ures.
\newblock A non-dynamically coherent example on $\mathbb{T}^3$.
\newblock Preprint 2010.

\bibitem{HHU07s}
F.~Rodriguez Hertz, M.~A.~Rodriguez Hertz, and R.~Ures.
\newblock A survey of partially hyperbolic dynamics.
\newblock In {\em Partially hyperbolic dynamics, laminations, and
  {T}eichm\"uller flow}, volume~51 of {\em Fields Inst. Commun.}, pages 35--87.
  Amer. Math. Soc., 2007.

\bibitem{HHU08b}
F.~Rodriguez Hertz, M.~A.~Rodriguez Hertz, and R.~Ures.
\newblock Accessibility and stable ergodicity for partially hyperbolic
  diffeomorphisms with 1{D}-center bundle.
\newblock {\em Invent. Math.}, 172:353--381, 2008.

\bibitem{HPS77}
M.~Hirsch, C.~Pugh, and M.~Shub.
\newblock {\em Invariant manifolds}, volume 583 of {\em Lect. Notes in Math.}
\newblock Springer Verlag, 1977.

\bibitem{Jo88}
J.-L. Journ{\'e}.
\newblock A regularity lemma for functions of several variables.
\newblock {\em Rev. Mat. Iberoamericana}, 4:187--193, 1988.

\bibitem{Ka80}
A.~Katok.
\newblock Lyapunov exponents, entropy and periodic points of diffeomorphisms.
\newblock {\em Publ. Math. IHES}, 51:137--173, 1980.

\bibitem{KK96}
A.~Katok and A.~Kononenko.
\newblock Cocycle's stability for partially hyperbolic systems.
\newblock {\em Math. Res. Lett.}, 3:191--210, 1996.

\bibitem{Le84a}
F.~Ledrappier.
\newblock Propri{\'e}t{\'e}s ergodiques des mesures de {S}ina{\"\i}.
\newblock {\em Publ. Math. I.H.E.S.}, 59:163--188, 1984.

\bibitem{Le86}
F.~Ledrappier.
\newblock Positivity of the exponent for stationary sequences of matrices.
\newblock In {\em Lyapunov exponents (Bremen, 1984)}, volume 1186 of {\em Lect.
  Notes Math.}, pages 56--73. Springer, 1986.

\bibitem{LY85a}
F.~Ledrappier and L.-S. Young.
\newblock The metric entropy of diffeomorphisms. {I}. {C}haracterization of
  measures satisfying {P}esin's entropy formula.
\newblock {\em Ann. of Math.}, 122:509--539, 1985.

\bibitem{LY85b}
F.~Ledrappier and L.-S. Young.
\newblock The metric entropy of diffeomorphisms. {I}{I}. {R}elations between
  entropy, exponents and dimension.
\newblock {\em Ann. of Math.}, 122:540--574, 1985.

\bibitem{Ox80}
J.~C. Oxtoby.
\newblock {\em Measure and category}, volume~2 of {\em Graduate Texts in
  Mathematics}.
\newblock Springer-Verlag, New York, 1980.

\bibitem{Pes76}
Ya.~B. Pesin.
\newblock Families of invariant manifolds corresponding to non-zero
  characteristic exponents.
\newblock {\em Math. USSR. Izv.}, 10:1261--1302, 1976.

\bibitem{Pes77}
Ya.~B. Pesin.
\newblock Characteristic {L}yapunov exponents and smooth ergodic theory.
\newblock {\em Russian Math. Surveys}, 324:55--114, 1977.

\bibitem{PSh72}
C.~Pugh and M.~Shub.
\newblock Ergodicity of {A}nosov actions.
\newblock {\em Invent. Math.}, 15:1--23, 1972.

\bibitem{PSh89}
C.~Pugh and M.~Shub.
\newblock Ergodic attractors.
\newblock {\em Trans. Amer. Math. Soc.}, 312:1--54, 1989.

\bibitem{PSh96}
C.~Pugh and M.~Shub.
\newblock Stable ergodicity and partial hyperbolicity.
\newblock In {\em International {C}onference on {D}ynamical {S}ystems
  ({M}ontevideo, 1995)}, volume 362 of {\em Pitman Res. Notes Math. Ser.},
  pages 182--187. Longman, 1996.

\bibitem{PSh04}
C.~Pugh and M.~Shub.
\newblock Stable ergodicity.
\newblock {\em Bull. Amer. Math. Soc.}, 41:1--41 (electronic), 2004.
\newblock With an appendix by Alexander Starkov.

\bibitem{PSW97}
C.~Pugh, M.~Shub, and A.~Wilkinson.
\newblock H{\"o}lder foliations.
\newblock {\em Duke Math. J.}, 86:517--546, 1997.

\bibitem{PVW}
C.~Pugh, M.~Viana, and A.~Wilkinson.
\newblock Absolute continuity of foliations.
\newblock In preparation.

\bibitem{RSS96}
D.~Repov{\v{s}}, A.~Skopenkov, and E.~{\v{S}}{\v{c}}epin.
\newblock {$C\sp 1$}-homogeneous compacta in {${\bf R}\sp n$} are {$C\sp
  1$}-submanifolds of {${\bf R}\sp n$}.
\newblock {\em Proc. Amer. Math. Soc.}, 124:1219--1226, 1996.

\bibitem{Ro52}
V.~A. Rokhlin.
\newblock On the fundamental ideas of measure theory.
\newblock {\em A. M. S. Transl.}, 10:1--52, 1952.
\newblock Transl. from Mat. Sbornik 25 (1949), 107--150.

\bibitem{RW01}
D.~Ruelle and A.~Wilkinson.
\newblock Absolutely singular dynamical foliations.
\newblock {\em Comm. Math. Phys.}, 219:481--487, 2001.

\bibitem{SW00}
M.~Shub and A.~Wilkinson.
\newblock Pathological foliations and removable zero exponents.
\newblock {\em Invent. Math.}, 139:495--508, 2000.

\bibitem{Almost}
M.~Viana.
\newblock Almost all cocycles over any hyperbolic system have nonvanishing
  {L}yapunov exponents.
\newblock {\em Ann. of Math.}, 167:643--680, 2008.

\bibitem{VY1}
M.~Viana and J.~Yang.
\newblock Physical measures and absolute continuity for one-dimensional center
  directions.
\newblock Preprint www.impa.br/$\sim$viana 2010.

\bibitem{Wliv}
A.~Wilkinson.
\newblock The cohomological equation for partially hyperbolic diffeomorphisms.
\newblock Preprint www.arXiv.org 2008.

\end{thebibliography}

\end{document}